\documentclass[11pt,a4]{article}

\newcommand{\N}{\ensuremath{\mathbb{N}}}

\usepackage{latexsym} 
\usepackage{theorem}
\usepackage{graphics}
\usepackage{graphicx}
\usepackage{amsmath}
\usepackage{amssymb}

\newtheorem{defeng}{Definition}[section]
\newtheorem{theorem}[defeng]{Theorem}

\newtheorem{lemma}[defeng]{Lemma}

\newtheorem{conjecture}[defeng]{Conjecture}

{\theorembodyfont{\rmfamily} }
{\theorembodyfont{\rmfamily} }
{\theorembodyfont{\rmfamily} }
{\theoremstyle{break}\theorembodyfont{\rmfamily} }
{\theoremstyle{break}\theorembodyfont{\rmfamily} }

\newcommand{\tp}{\!-\!}

\newcounter{claim}

\newenvironment{proof}[1][]%
 {\noindent {\setcounter{claim}{0}\sc proof ---
   }{#1}{}}{\hfill$\Box$\vspace{2ex}}

\def\claim{$$\vcenter\bgroup\advance\hsize by -8em\noindent
\refstepcounter{claim}\ignorespaces\it}	    
\makeatletter
\def\endclaim{\rm\egroup\leqno(\theclaim)$$\global\@ignoretrue}
\makeatother

\newenvironment{proofclaim}[1][]%
	{\noindent {}{#1}{}}{ This proves claim~(\arabic{claim}).\vspace{1ex}}

\newcommand{\sm}{\setminus} 

\title{On graphs with no induced subdivision of $K_4$}

\author{ Benjamin L\'ev\^eque\thanks{CNRS, LIRMM, 161 rue Ada,
    34392~Montpellier~Cedex~05, France, benjamin.leveque@lirmm.fr}~,
Fr\'ed\'eric Maffray\thanks{CNRS, Laboratoire G-SCOP, 46 avenue
F\'{e}lix Viallet, 38031~Grenoble~Cedex, France,
frederic.maffray@g-scop.inpg.fr}~ and 
Nicolas Trotignon\thanks{CNRS, LIP, ENS Lyon, INRIA, Universit\'e de Lyon, France,
  nicolas.trotignon@ens-lyon.fr. Partially supported by \emph{Agence Nationale de la Recherche} under reference
    \textsc{anr 10 jcjc 0204 01} and by PHC Pavle Savi\'c grant,
    jointly awarded by EGIDE, an agency of the French Minist\`ere des
    Affaires \'etrang\`eres et europ\'eennes, and Serbian Ministry for
    Science and Technological Development.}}

\date{April 19, 2012}

\begin{document}

\maketitle

\section*{Abstract}
We prove a decomposition theorem for graphs that do not contain a
subdivision of $K_4$ as an induced subgraph where $K_4$ is the
complete graph on four vertices.  We obtain also a structure theorem
for the class $\cal C$ of graphs that contain neither a subdivision of
$K_4$ nor a wheel as an induced subgraph, where a wheel is a cycle on
at least four vertices together with a vertex that has at least three
neighbors on the cycle.  Our structure theorem is used to prove that
every graph in $\cal C$ is 3-colorable and entails a polynomial-time
recognition algorithm for membership in $\cal C$.  As an intermediate
result, we prove a structure theorem for the graphs whose cycles are
all chordless.

{\bf\noindent AMS Classification: } 05C75

\section{Introduction}

We use the standard notation from~\cite{bondy.murty:book}.  Unless
otherwise specified, we say that a graph $G$ \emph{contains} $H$ when
$H$ is isomorphic to an induced subgraph of $G$.  Denote by $K_4$ the
complete graph on four vertices.  A subdivision of a graph $G$ is
obtained by subdividing edges of $G$ into paths of arbitrary length
(at least one).  We say that $H$ is an \emph{ISK4 of a graph $G$} when
$H$ is an induced subgraph of $G$ and $H$ is a subdivision of $K_4$.
A graph that does not contain any subdivision of $K_4$ is said to be
\emph{ISK4-free}.  Our main result is Theorem~\ref{th:main}, saying
that every ISK4-free graph is either in some basic class or has some
special cutset.  In~\cite{leveque.lmt:detect}, it is mentioned that
deciding in polynomial time whether a given graph is ISK4-free is an
open question of interest.  This question was our initial motivation.
But our theorem does not lead to a polynomial-time recognition
algorithm so far.  The main reason is that at some step we use cutsets
(namely star cutsets and double star cutsets) that are difficult to
use in algorithms.  We leave as an open question the existence of a
more powerful decomposition theorem.

A consequence of our work is a complete structural description of the
class ${\cal C}$ of graphs that contain no ISK4 and no wheel.  Note
that this class is easily seen to be the class of graphs with no
$K_4$ and subdivision of a wheel as an induced subgraph.  We give
a recognition algorithm for this class, a coloring algorithm, and we
prove that every graph in this class is 3-colorable.

Before stating our main results more precisely, we introduce some
definitions and notation.

A \emph{hole} of a graph is an induced cycle on at least four
vertices.  A \emph{wheel} is a graph that consists of a hole $H$ plus
a vertex $x\notin H$, called the \emph{hub} of the wheel, that is
adjacent to at least three vertices of the hole.  An edge of the wheel
that is incident to $x$ is called a \emph{spoke}.  A vertex $v$ of a
graph is \emph{complete} to a set of vertices $S\subseteq
V(G)\setminus v$ if $v$ is adjacent to every vertex in $S$.  A vertex
$v$ is \emph{anticomplete} to a set of vertices $S$ if $v$ is adjacent
to no vertex in $S$.  Two disjoint sets $A,B$ are \emph{complete} to
each other if every vertex of $A$ is complete to $B$.  A graph is
called \emph{complete bipartite} (resp.~\emph{complete tripartite}) if
its vertex-set can be partitioned into two (resp.~three) non-empty
stable sets that are pairwise complete to each other.  If these two
(resp.~three) sets have size $p,q$ (resp.~$p, q, r$) then the graph is
denoted by $K_{p, q}$ (resp.~$K_{p,q,r}$).

Given a graph $H$, the \emph{line graph} of $H$ is the graph $L(H)$
with vertex-set $E(G)$ and edge-set $\{ef\, : \, e\cap f\neq
\emptyset\}$.  The graph $H$ is called a \emph{root} of $L(H)$.

We denote the path on vertices $x_1, \ldots, x_n$ with edges $x_1x_2$,
$\ldots$, $x_{n-1}x_n$ by $x_1 \tp \cdots \tp x_n$.  We also say that
$P$ is a $(x_1,x_n)$-path.  We denote by $x_i-P-x_j$ the subpath of
$P$ with extremities $x_i, x_j$.  A path or a cycle is
\emph{chordless} if it is an induced subgraph of the graph that we are
working on.

Given two graphs $G, G'$, we denote by $G\cup G'$ the graph whose
vertex set is $V(G)\cup V(G')$ and whose edge set is $E(G) \cup
E(G')$.

For any integer $k\ge 0$, a \emph{$k$-cutset} in a graph is a subset
$S\subset V(G)$ of size $k$ such that $G\sm S$ is disconnected.  A
\emph{proper $2$-cutset} of a graph $G$ is a $2$-cutset $\{a, b\}$
such that $ab\notin E(G)$, $V(G)\sm \{a,b \}$ can be partitioned into
two non-empty sets $X$ and $Y$ so that there is no edge between $X$
and $Y$ and each of $G[X \cup \{ a,b \}]$ and $G[Y \cup \{ a,b \}]$ is
not an $(a,b)$-path.

A \emph{star-cutset} of a graph is a set $S$ of vertices such that
$G\sm S$ is disconnected and $S$ contains a vertex adjacent to every
other vertex of $S$.

A \emph{double star cutset} of a graph is a set $S$ of vertices such
that $G\sm S$ is disconnected and $S$ contains two adjacent vertices
$u, v$ such that every vertex of $S$ is adjacent at least one of $u,
v$.  Note that a star-cutset is either a double star cutset or
consists of one vertex.

A multigraph is called \emph{series-parallel} if it arises from a
forest by applying the following operations repeatedly: adding a
parallel edge to an existing edge; subdividing an edge.  A
\emph{series-parallel graph} is a series-parallel multigraph with no
parallel edges.

Our main result is the following, which is proved in
Section~\ref{sec:th:main}.

\begin{theorem}
  \label{th:main}
  Let $G$ be an ISK4-free graph. Then either:
  \begin{itemize}
  \item  
    $G$ is series-parallel; 
  \item  
    $G$ is the line graph of a graph with maximum degree at most
    three;
  \item  
    $G$ has clique-cutset, a proper $2$-cutset, a star-cutset or a double
    star cutset.
  \end{itemize}
\end{theorem}

The proof of the theorem above follows a classical idea.  We consider
a basic graph $H$ and prove that if a graph in our class contains $H$,
then either the whole graph is basic, or some part of the graph
attaches to $H$ in a way that entails a decomposition.  Then, for the
rest of the proof, the graphs under consideration can be considered
$H$-free.  We consider another basic graph $H'$, and so on.  The basic
graphs that we consider are $K_{3, 3}$, then some substantial
line graph, then prisms, and finally the octahedron and wheels.  The
idea of considering a maximal line graph in such a context was first
used in~\cite{conforti.c:wp}.  The same idea is essential in proof of
the Strong Perfect Graph Conjecture \cite{chudnovsky.r.s.t:spgt}.

Given a graph $G$, an induced subgraph $K$ of $G$, and a set $C$ of
vertices of $G\setminus K$, the \emph{attachment} of $C$ over $K$ is
$N(C) \cap V(K)$, which we also denote by $N_K(C)$.  When a set $S$ =
$\{u_1, u_2, u_3, u_4\}$ induces a square in a graph $G$ with $u_1,
u_2, u_3, u_4$ in this order along the square, a \emph{link} of $S$ is
an induced path $P$ of $G$ with ends $p, p'$ such that either $p=p'$
and $N_S(p) = S$, or $N_S(p) = \{u_1, u_2\}$ and $N_S(p') = \{u_3, u_4
\}$, or $N_S(p) = \{u_1, u_4\}$ and $N_S(p') = \{u_2, u_3\}$, and no
interior vertex of $P$ has a neighbor in $S$.  A link with ends $p,
p'$ is said to be \emph{short} if $p=p'$, and \emph{long} if $p\neq
p'$.  A \emph{rich square} (resp.\ \emph{long rich square}) is a graph
$K$ that contains a square $S$ as an induced subgraph such that
$K\setminus S$ has at least two components and every component of
$K\setminus S$ is a link (resp.\ a long link) of~$S$.  Then $S$ is
called a central square of $K$.  A rich square may have several
central squares; for example $K_{2,2,2}$ is a rich square with three
central squares.

In the particular case of wheel-free graph we have the following
structure theorem.  Note that a rich square is wheel-free if and only
if it is long.  A graph is \emph{chordless} if all its cycles are
chordless.  It is easy to check that a line graph $G=L(R)$ is
wheel-free if and only if $R$ is chordless.

\begin{theorem}
  \label{th:nowheel}
  Let $G$ be an \{ISK4, wheel\}-free graph. Then either:
  \begin{itemize}
  \item $G$ is series-parallel;
  \item $G$ is the line graph of a chordless graph with maximum degree
    at most three;
  \item $G$ is a complete bipartite graph;
  \item $G$ is a long rich square;
  \item $G$ has clique-cutset or a proper $2$-cutset.
  \end{itemize}
\end{theorem}
The structure of chordless graphs is elucidated in the following
theorem, which will be proved in Section~\ref{sec:thm:nochord}.  Let
us say that a graph $G$ is \emph{sparse} if for every edge $uv$ of $G$
we have either $\deg(u) \leq 2$ or $\deg(v) \leq 2$.  
\begin{theorem}\label{thm:nochord}
  Let $G$ be a chordless graph.  Then either $G$ is sparse or $G$
  admits a $1$-cutset or a proper $2$-cutset.
\end{theorem}

Theorems~\ref{th:nowheel} and \ref{thm:nochord} can be used to derive
a tight bound on the chromatic number of \{ISK4, wheel\}-free graphs.
\begin{theorem}
  \label{l:WF3c}
  Any \{ISK4, wheel\}-free graph is 3-colorable. 
\end{theorem}
Theorem~\ref{l:WF3c} will be proved in Section~\ref{s:WF3c}.  This
theorem is tight as shown by the graph on Figure~\ref{fig:C5PlusV}.

\begin{figure}[h]
  \centering
  \includegraphics[scale=0.5]{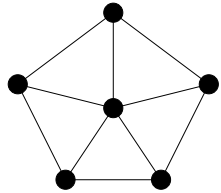}
  \caption{Example of an ISK4-free graph with chromatic number 4}
  \label{fig:C5PlusV}
\end{figure}

Gy\'arf\'as~\cite{gyarfas:perfect} defines a graph $G$ to be
\emph{$\chi$-bounded} with \emph{$\chi$-bounding function $f$} if for
all induced subgraphs $G'$ of $G$ we have $\chi(G') \leq
f(\omega(G'))$.  A class of graphs is \emph{$\chi$-bounded} if there
exists a $\chi$-bounding function that holds for all graphs of the
class.  Scott~\cite{scott:tree} conjectured that for any graph $H$,
the class of those graphs that do not contain any subdivision of $H$
as an induced subgraph is $\chi$-bounded.  This conjectured was
disproved by Pawlick et al.\ \cite{pawlikEtAl:scott}.  It still remains
to figure out for which $H$'s the statement conjectured by Scott is
true.  As noted by Scott~\cite{scott:pc}, some of our results can be
combined with a theorem of K\"uhn and Osthus~\cite{kuhnOsthus:04} to
prove his conjecture in the particular case of $K_4$.  Note that being
$\chi$-bounded for the class of ISK4-free graphs means having the
chromatic number bounded by a constant (because $K_4$ is a particular
ISK4).

\begin{theorem}[K\"uhn and Osthus~\cite{kuhnOsthus:04}]
  \label{th:KO}
  For every graph $H$ and every $s \in \N$ there exists $d = d(H, s)$ such
  that every graph $G$ of average degree at least $d$ contains either a
  $K_{s, s}$ as a subgraph or an induced subdivision of $H$.
\end{theorem}

\begin{theorem}[Scott \cite{scott:pc}]
  \label{th:ScottK4}
  There exists a constant $c$ such that any ISK4-free graphs is
  $c$-colorable.
\end{theorem}

Theorem~\ref{th:ScottK4} will be proved in Section~\ref{sec:K33}.  In
fact, we do not know any example of an ISK4-free graph whose chromatic
number is 5 or more.  We propose the following conjecture.

\begin{conjecture}
  Any ISK4-free graph is 4-colorable. 
\end{conjecture}

Our results yield several algorithms described in Section~\ref{sec:algo}.

\begin{theorem}\label{th:algo}
  There exists an algorithm of complexity $O(n^2m)$ that decides whether
  a given graph is \{ISK4, wheel\}-free.

  There exists an algorithm of complexity $O(n^2m)$ whose input is a
  graph with no ISK4 and no wheel and whose output is a 3-coloring of
  its vertices.  
\end{theorem}


\section{Series-parallel graphs}

\begin{theorem}[Duffin~\cite{duffin:SP}, Dirac~\cite{dirac:SP}]
  \label{th:dirac}
  A graph is series-parallel if and only if it contains no
  subdivision of $K_4$ as a (possibly non-induced) subgraph.
\end{theorem}

A \emph{branch-vertex} in a graph $G$ is a vertex of degree at
least~3.  A \emph{branch} is a path of $G$ of length at least one
whose ends are branch-vertices and whose internal vertices are not (so
they all have degree~2).  Note that a branch of $G$ whose ends are $u,
v$ has at most one chord: $uv$.  An induced subdivision $H$ of $K_4$
has four vertices of degree three, which we call the \emph{corners}
of $H$, and six branches, one for each pair of
corners.

A \emph{theta} is a connected graph with exactly two vertices of
degree three, all the other vertices of degree two, and three
branches, each of length at least two.  A \emph{prism} is a graph that
is the line graph of a theta.

\begin{lemma}
  \label{l:begin}
  Let $G$ be an ISK4-free graph.  Then either $G$ is a series-parallel
  graph, or $G$ contains a prism, a wheel or a $K_{3, 3}$.
\end{lemma}

\begin{proof}
  Suppose that $G$ is not series-parallel.  By Theorem~\ref{th:dirac},
  $G$ contains a subdivision $H$ of $K_4$ as a possibly non-induced
  subgraph.  Let us choose a minimal such subgraph $H$.  So $H$ can be
  obtained from a subdivision $H'$ of $K_4$ by adding edges (called
  \emph{chords}) between the vertices of $H'$.  Since $G$ is
  ISK4-free, there is at least one such chord $e$ in $H$.  Let $H'$
  have corners $a, b, c, d$ and branches $P_{ab}$, $P_{ac}$, $P_{ad}$,
  $P_{bc}$, $P_{bd}$, $P_{cd}$ with the obvious notation.  Note that,
  by the minimality of $H$, the six paths $P_{ab}$, $P_{ac}$,
  $P_{ad}$, $P_{bc}$, $P_{bd}$, $P_{cd}$ are chordless in $H$.

  Suppose that $e$ is incident to one of $a, b, c, d$, say $e= ax$.
  Then $x$ lies in none of $P_{ab}$, $P_{ac}$, $P_{ad}$ by the
  minimality of $H$.  Moreover $P_{ab}$, $P_{ac}$, $P_{ad}$ have all
  length one, for otherwise, by deleting the interior vertices of one
  of them, we obtain a subdivision of $K_4$, which contradicts the
  minimality of $H$.  If $H$ has a chord $e'$ that is not incident to
  $a$, then $e'$ is a chord of the cycle $C = P_{bd} \cup P_{cd} \cup
  P_{bc}$.  Since $C$ is a cycle with one chord $e'$ and since the
  branches $P_{bd}, P_{cd}, P_{bc}$ are chordless, we may assume up to
  symmetry that $C$ contains a cycle $C'$ that goes through $e', c, d$
  and not through $b$.  If $x$ is in $C'$, then $C'\cup \{a\}$ is a
  subdivision of $K_4$, which contradicts the minimality of $H$.  So,
  up to the symmetry between $P_{bc}$ and $P_{bd}$, we may assume that
  $x$ is in $P_{bd}\sm C'$.  Then $C' \cup x\tp P_{bd} \tp d\cup
  \{a\}$ forms a subdivision of $K_4$, which contradicts the the
  minimality of $H$.  Hence, every chord of $H$ is incident to $a$.
  This means that $H$ is a wheel with hub $a$ and the lemma holds.
  From now on, we assume that no chord of $H$ is adjacent to $a, b, c,
  d$.

  Suppose that $e$ is between interior vertices of two branches of $H$
  with a common end, $P_{ab}$ and $P_{ad}$ say.  Put $e=uv$, where
  $u\in P_{ab}$, $v\in P_{ad}$.  Vertices $a$ and $u$ are adjacent,
  for otherwise the deletion of the interior vertices of $a \tp P_{ab}
  \tp u$ produces a subdivision of $K_4$, which contradicts the
  minimality of $H$.  Similarly, $a$ and $v$ are adjacent, and
  $P_{bc}$, $P_{bd}$, $P_{cd}$ all have length one.  So $H'$ is a
  prism, whose triangles are $auv, bcd$.  If $H=H'$, the lemma holds,
  so let us assume that $H'\neq H$.  Then $H$ has a chord $e'$ that is
  not an edge of $H'$.  Up to symmetry, we assume that $e'$ has an end
  $u'$ in $u P_{ab} b$ and an end $v'$ in $v P_{ad} d$.  Note that
  $u'\neq b$ and $v'\neq d$.  Since $e\neq e'$ we may assume $u \neq
  u'$.  Then the deletion of the interior vertices of $aP_{ab}u'$
  gives a subdivision of $K_4$, which contradicts the minimality of
  $H$.   

  Finally, suppose that $e$ is between two branches of $H$ with no
  common end, $P_{ad}$ and $P_{bc}$ say.  Put $e=uv$, $u\in P_{ad}$,
  $v\in P_{bc}$.  If $P_{ab}$ has length greater than one , then by
  deleting its interior we obtain a subdivision of $K_4$, which
  contradicts the minimality of $H$.  So, $P_{ab}$, and similarly
  $P_{ac}$, $P_{bd}$, $P_{cd}$, all have length one.  The same
  argument shows that $ua$, $ud$, $vb$, $vc$ are edges of $H$.  Hence
  $H$ is isomorphic to $K_{3, 3}$.
\end{proof}


\section{Complete bipartite graphs}
\label{sec:K33}

Here we decompose ISK4-free graphs that contain a $K_{3, 3}$.

\begin{lemma}
  \label{l.k33n}
  Let $G$ be an ISK4-free graph, and $H$ be a maximal induced $K_{p,
  q}$ in $G$, such that $p, q \geq 3$.  Let $v$ be a vertex of $G\sm
  H$.  Then the attachment of $v$ over $H$ is either empty, or
  consists of one vertex or of one edge or is $V(H)$.
\end{lemma}

\begin{proof}
Let $A=\{a_1, \dots, a_p\}$ and $B = \{b_1, \dots, b_q\}$ be the two
sides of the bipartition of $H$.  If $v$ is adjacent to at most one
vertex in $A$ and at most one in $B$, then the lemma holds.  Suppose
now, up to symmetry, that $v$ is adjacent to at least two vertices in
$A$, say $a_1, a_2$.  Then $v$ is either adjacent to every vertex in
$B$ or to no vertex in $B$, for otherwise, up to symmetry, $v$ is
adjacent to $b_1$ and not to $b_2$, and $\{a_1, a_2, b_1, b_2, v\}$ is
an ISK4.  If $v$ has no neighbor in $B$, then $v$ sees every vertex in
$A$, for otherwise $v a_3 \notin E(G)$ say, and $\{a_1, a_2, a_3, b_1,
b_2, v\}$ is an ISK4.  So, $v$ is complete to $A$ and anticomplete to
$B$, which contradicts the maximality of $H$.  If $v$ is complete to
$B$, then $v$ is adjacent to at least two vertices in $B$ and
symmetrically we can prove that $v$ is complete to $A$.  So, the
attachment of $v$ is $V(H)$.
 \end{proof}

\begin{lemma}
  \label{l.k33C}
  Let $G$ be an ISK4-free graph that contains a $K_{3,3}$, and let $H$
  be a maximal induced $K_{p, q}$ of $G$ with $p, q \geq 3$.  Let $U$
  be the set of those vertices of $V(G)\setminus H$ that are complete
  to $H$.  Let $C$ be a component of $G\setminus(H\cup U)$.  Then the
  attachment of $C$ over $H$ is either empty or consists of one vertex
  or of one edge.
\end{lemma}

\begin{proof}
  Suppose the contrary.  So we may assume up to symmetry that there
  are vertices $c_1, c_2$ in $C$ such that $|N(\{c_1, c_2\}) \cap D|
  \geq 2$ where $D$ is one of $A, B$.  Since $C$ is connected, there
  is a path $P = c_1 \tp \cdots \tp c_2$ in $C$ from $c_1$ to $c_2$.
  We choose $c_1, c_2$ such that $P$ is minimal.  Up to symmetry, we
  may assume that $c_1a_1, c_2a_2 \in E(G)$.  By Lemma~\ref{l.k33n},
  we have $c_1\neq c_2$.  If $a_3$ has a neighbor in $P$, then by
  Lemma~\ref{l.k33n} this neighbor must be an interior vertex of $P$,
  but this contradicts the minimality of $P$.  So, $a_3$ has no
  neighbor in $P$.  If no vertex in $B$ has neighbors in $P$, then
  $V(P) \cup \{a_1, a_2, a_3, b_1, b_2\}$ induces an ISK4.  If exactly
  one vertex in $B$, say $b_1$, has neighbors in $P$, then $V(P)\cup
  \{a_1, a_2, a_3, b_2, b_3\}$ induces an ISK4.  If at least two
  vertices in $B$, say $b_1, b_2$, have neighbors in $P$, then by
  Lemma~\ref{l.k33n} and by the minimality of $P$ we may assume that
  $N(b_1) \cap V(P) = \{c_1\}$ and $N(b_2) \cap V(P) = \{c_2\}$.  But
  then $V(P)\cup \{a_1, a_3, b_1, b_2\}$ induces an ISK4.  In every
  case there is a contradiction.
\end{proof}

Let us say that a complete bipartite or complete tripartite graph is
\emph{thick} if it contains an induced $K_{3,3}$.

\begin{lemma}
  \label{l:decK33}
  Let $G$ be an ISK4-free graph that contains $K_{3, 3}$.  Then either
  $G$ is a thick complete bipartite or complete tripartite graph, or
  $G$ has a clique-cutset of size at most three.
\end{lemma}

\begin{proof}
  Let $H$ be a maximal $K_{p, q}$ in $G$, with $p, q \geq 3$, and let
  $U$ be the set of those vertices that are complete to $H$.  Note
  that $U$ is a stable set because if $U$ contains an edge $uv$, then
  $\{u, v, a_1, b_1\}$ is an ISK4.  If $V(G) = V(H)\cup U$, then $G$
  is either a complete bipartite graph (if $U = \emptyset$) or
  complete tripartite graph (if $U\neq \emptyset$).  Now suppose that
  $V(G) \neq V(H)\cup U$, and let $C$ be any component of
  $G\setminus(H\cup U)$.  We claim that $|N(C)\cap U|\leq 1$.  Else,
  consider two vertices $u, v$ in $N(C) \cap U$ and a minimal path $P$
  in $C$ from a neighbor of $u$ to a neighbor of $v$.  By
  Lemma~\ref{l.k33C}, we may assume that $a_3$ and $b_3$ have no
  neighbor in $C$ (hence in $P$).  Then $P\cup \{u, v, a_3, b_3\}$ is
  an ISK4, a contradiction.  This proves our claim.  By
  Lemma~\ref{l.k33C}, $N(C)\cap(V(H)\cup U)$ is a clique-cutset of $G$
  of size at most three.
\end{proof}

%
%
%
%

\subsubsection*{Proof of Theorem~\ref{th:ScottK4}}

Let $c=d(K_4,6)\geq 3$ the constant of Theorem~\ref{th:KO} with
$H=K_4$ and $s=6$.  We claim that any ISK4-free graphs is
$c$-colorable.  Suppose on the contrary that there exists an ISK4-free
graph $G$ with $\chi(G)> c$, and suppose $G$ is minimal with this
property, i.e. $\chi(H)\leq c$ for every proper induced subgraph $H$
of $G$.

We claim that the degree of every vertex is at least $c$.  Suppose on
the contrary that $G$ contains a vertex $v$ of degree $deg(v)\leq
c-1$, then $\chi(G)\leq max(\chi(G-v),deg(v)+1)\leq c$, a
contradiction.  So the average degree of $G$ is at least $c=d(K_4,6)$

By Theorem~\ref{th:KO} the graph $G$ contains a $K_{6,6}$ as a
possibly non-induced subgraph. Let $A,B$ be the two side of the
$K_{6,6}$. The graph $G[A]$ contains no triangle, otherwise this
triangle plus a vertex of $B$ forms a $K_4$. Similarly $G[B]$ contains
no triangle. From the well known fact that any graph on 6 vertices
contains either a triangle or a stable set on 3 vertices, both $G[A]$
and $G[B]$ contains a stable set of size 3. So $G$ contains an induced
$K_{3,3}$.

By Lemma~\ref{l:decK33}, the graph $G$ admits a clique cutset $K$.
Hence $V(G)\setminus K$ is partitioned into non-empty sets $X_1, X_2$
such that there are no edges between $X_1$ and $X_2$.  A coloring of
$G$ can be easily obtained by combining a coloring of $G[K\cup X_1]$
and $G[K\cup X_2]$, showing that $\chi(G)\leq max(\chi(G[K\cup
X_1]),\chi(G[K\cup X_2]))\leq c$.


\section{Cyclically $3$-connected graphs}

A \emph{separation} of a graph $H$ is a pair $(A, B)$ of subsets of
$V(H)$ such that $A \cup B = V(H)$ and there are no edges between $A
\setminus B$ and $B \setminus A$.  It is \emph{proper} if both $A
\setminus B$ and $B\setminus A$ are non-empty.  The \emph{order} of
the separation is $|A\cap B|$.  A \emph{$k$-separation} is a
separation $(A, B)$ such that $|A \cap B| \leq k$.  A separation $(A,
B)$ is \emph{cyclic} if both $H[A]$ and $H[B]$ has cycles.  A graph
$H$ is \emph{cyclically $3$-connected} if it is $2$-connected, not a
cycle, and there is no cyclic $2$-separation.  Note that a cyclic
$2$-separation of any graph is proper.

Here we state simple lemmas about cyclically $3$-connected graphs that
will be needed in the next section.  Most of them are stated and
proved implicitly in~\cite[Section~7]{chudnovsky.seymour:claw4}.  But
they are worth stating separately here: they are needed for the second
time at least and writing down their proof now may be convenient for
another time.  A cyclically $3$-connected graph has at least four
vertices and $K_4$ is the only cyclically $3$-connected graph on four
vertices.  As any $2$-connected graph that is not a cycle, a cyclically
$3$-connected graph is edge-wise partitioned into its branches.

\begin{lemma}
  \label{l:c3ccutset}
  Let $H$ be a cyclically $3$-connected graph.  For every proper
  $2$-separation $(A, B)$ of $H$, $A \cap B$ consists of two
  non-adjacent vertices, one of $H[A], H[B]$ is a path, and thus is
  included in a branch of $H$, and the other contains a cycle.
\end{lemma} 

\begin{proof}
  Since $(A, B)$ is proper, $A\cap B$ is a cutset, and so it has size
  two since $H$ is $2$-connected.  We put $A\cap B = \{a, b\}$.  Since
  $(A, B)$ is not cyclic, up to symmetry, $H[A]$ has no cycle.  Note
  that $H[A]$ contains a path $P$ from $a$ to $b$, for otherwise one
  of $a, b$ is a cutvertex of $H$, which contradicts $H$ being
  $2$-connected.  Actually, $H[A] = P$, for otherwise $H[A]$ is a tree
  with at least one vertex $c$ of degree~3, and $c$ is a cutvertex of
  this tree, so $c$ is also a cutvertex of $H$, a contradiction again.
  We have $ab\notin E(H)$ because $(P, B)$ is proper.  Since $(P, B)$
  is a separation, every internal vertex of $P$ has degree two in $H$,
  so $P$ is included in a branch of $H$ as claimed.  So, $ab\notin
  E(H)$ because $(P, B)$ is proper.  If $B$ has no cycle, then by the
  same proof as for $A$, $H[B]$ is a path.  So, $H$ is a cycle, a
  contradiction.
\end{proof}

\begin{lemma}
  \label{l:2cutset}
  Let $H$ be a cyclically $3$-connected graph and $a, b$ be two
  adjacent vertices of $H$.  Then $\{a, b\}$ is not a cutset of $H$.
\end{lemma}

\begin{proof}
  Follows directly from Lemma~\ref{l:c3ccutset}.
\end{proof}

\begin{lemma}
  \label{l:thetask4}
  Let $H$ be a cyclically $3$-connected graph, $a, b$ be two
  branch-vertices of $H$, and $P_1, P_2, P_3$ be three induced paths
  of $H$ whose ends are $a, b$.  Then either:

  \begin{itemize}
  \item 
    $P_1, P_2, P_3$ are branches of $H$ of length at least two and $H
    = P_1 \cup P_2 \cup P_3$ (so $H$ is a theta);
  \item 
    there exist distinct integers $i, j \in \{1, 2, 3\}$ and a path
    $S$ of $H$ with an end in the interior of $P_i$, an end in the
    interior of $P_j$ and whose interior is disjoint from $V(P_1\cup
    P_2 \cup P_3)$; and $P_1\cup P_2 \cup P_3 \cup S$ is a subdivision
    of $K_4$.
  \end{itemize}
\end{lemma}

\begin{proof}
  Put $H' = P_1 \cup P_2 \cup P_3$.  Suppose that $H = H'$.  If $P_1$
  is of length one, then $(V(P_1 \cup P_2), V(P_1 \cup P_3))$ is a
  cyclic $2$-separation of $H$.  So $P_1$, and similarly $P_2, P_3$
  are of length at least two and the first outcome of the lemma holds.
  Now assume $H \neq H'$.  If the second outcome of the lemma fails,
  then no path like $S$ exists.  In particular there is no edge
  between the interior of any two of the three paths, and the
  interiors of the three paths lie in distinct components of $H
  \setminus \{a, b\}$.  Since $H$ is connected and $H \neq H'$, there
  is a vertex in $V(H) \setminus V(H')$ with a neighbor $c$ in one of
  $P_1, P_2, P_3$.  Since $H$ is $2$-connected, $\{c\}$ is not a
  cutset of $H$ and there exists a path $R$ from $c$ to some other
  vertex $c'$ in $H'$.  Since no path like $S$ exists, $R$ must have
  its two ends in the same branch of $H'$, say in $P_1$.  It follows
  that $P_1$ has an interior vertex, and we call $C$ the component of
  $H\sm \{a, b\}$ that contains the interior of $P_1$ union the
  component that contains the interior of $R$.  Now, we put $A = \{a,
  b \} \cup V(H) \sm C$, $B=C\cup \{a, b\}$ and we observe that $(A,
  B)$ is a cyclic $2$-separation of $H$, a contradiction.
\end{proof}

\begin{lemma}
  \label{l:2implytheta}
  Let $H$ be a cyclically $3$-connected graph and let $a, b$ be two
  branch-vertices of $H$ such that there exist two distinct branches
  of $G$ between them.  Then $H$ is a theta.
\end{lemma}

\begin{proof}
    Let $P_1, P_2$ be two distinct branches of $H$ whose ends are $a,
    b$.  Put $A = V(P_1 \cup P_2)$, $B = (V(H) \sm A)\cup\{a, b\}$,
    and observe that $(A, B)$ is a $2$-separation of $H$.  Since $H$ is
    not a cycle, $B$ contains at least three vertices, and $H[B]$
    contains a shortest path $P_3$ from $a$ to $b$ since $H$ is
    $2$-connected.  We apply Lemma~\ref{l:thetask4} to $P_1, P_2, P_3$.
    Since $P_1, P_2$ are branches, the second outcome cannot happen.
    So $H$ is a theta.
\end{proof}

\begin{lemma}
  \label{l:descc3c}
  A graph $H$ is cyclically $3$-connected if and only if it is either
  a theta or a subdivision of a $3$-connected graph.
\end{lemma}

\begin{proof}
  A $3$-connected graph has at least four vertices.  So, thetas and
  subdivisions of $3$-connected graphs are cyclically $3$-connected.
  Conversely, if $H$ is a cyclically $3$-connected graph, then let
  $H'$ be the multigraph on the branch-vertices of $H$ obtained as
  follows: for every branch of $H$ with ends $a, b$, we put an edge
  $ab$ in $H'$.  If $H'$ has a multiple edge, then there are two
  vertices $a, b$ of $H$ and two branches $P, Q$ of $H$ with ends $a,
  b$.  So, by Lemma~\ref{l:2implytheta}, $H$ is a theta.  Now assume
  that $H'$ has no multiple edge.  Then $H'$ is a graph and $H$ is a
  subdivision of $H'$.  Since $H$ is $2$-connected, $H'$ is also
  $2$-connected.  We claim that $H'$ is $3$-connected.  For suppose
  that $H'$ has a proper $2$-separation $(A, B)$.  Since $H'$ has
  minimum degree at least three, it is impossible that $H'[A]$ is a
  path.  Since $H'$ is $2$-connected, $H'[A]$ cannot be a tree and so
  it must contain a cycle.  Symmetrically, $H'[B]$ must contain a
  cycle.  Let $A'$ be the union of $A$ and of the set of vertices of
  degree two of $H$ that arise from subdividing edges of $H'[A]$.  Let
  $B'$ be defined similarly.  If $H'[A\cap B]$ is an edge and vertices
  of $H$ arise from the subdivision of that edge, we put them in $A'$.
  Now we observe that $(A', B')$ is a cyclic $2$-separation of $H$, a
  contradiction.  This proves our claim.  It follows that $H$ is a
  subdivision of a $3$-connected graph.
\end{proof}

\begin{lemma}
  \label{l:addedge}
 Let $H$ be a cyclically $3$-connected graph and $a, b$ be two
 distinct vertices of $H$.  If no branch contains both $a, b$, then
 $H' = (V(H), E(H) \cup\{ab\})$ is a cyclically $3$-connected graph
 and every graph obtained from $H'$ by subdividing $ab$ is cyclically
 $3$-connected.
\end{lemma}

\begin{proof}
  The graph $H'$ is clearly $2$-connected and not a cycle.  So we need
  only prove that $H'$ has no cyclic $2$-separation.  Suppose it has a
  cyclic $2$-separation $\{A, B\}$.  Up to symmetry we may assume that
  $a, b$ lie in $A$, because there is no edge between $A\setminus B$
  and $B\setminus A$.  Since $(A, B)$ is cyclic in $H'$, $B$ has a
  cycle in $H'$ and so in $H$.  Hence, by Lemma~\ref{l:c3ccutset}, $A$
  induces a path of $H$ and so it is included in a branch of $H$,
  contrary to our assumption.

  By Lemma~\ref{l:descc3c}, $H'$ is a subdivision of a $3$-connected
  graph since it cannot be a theta because of the edge $ab$. So, every
  graph that we obtain by subdividing $ab$ is a subdivision of a
  $3$-connected graph, and so is cyclically $3$-connected.   
\end{proof}

\begin{lemma}
  \label{l:2edges}
  Let $H$ be a cyclically $3$-connected graph, let $Z$ be a cycle of
  $H$ and $a, b, c, d$ be four distinct vertices of $Z$ that lie in
  this order on $Z$ and such that $ab \in E(Z)$ and $cd \in E(Z)$.
  Let $P$ be the subpath of $Z$ from $a$ to $d$ that does not contain
  $b, c$, and let $Q$ be the subpath of $Z$ from $b$ to $c$ that does
  not contain $a, d$.  Suppose that the edges $ab$, $cd$ are in two
  distinct branches of $H$.  Then there is a path $R$ with an
  end-vertex in $P$, an end-vertex in $Q$, no interior vertex in $Z$,
  and $R$ is not from $a$ to $b$ or from $c$ to $d$.
\end{lemma}

\begin{proof}
 Suppose there does not exist a path like $R$.  Then $\{a, c\}$ is a
 cutset of $H$ that separates $b$ from $d$.  By
 Lemma~\ref{l:c3ccutset}, we may assume up to symmetry that $a \tp P
 \tp d \tp c$ is included in a branch of $H$.  Also $\{b, d\}$ is a
 cutset, so one of $b \tp a \tp P \tp d$, $b \tp Q \tp c \tp d$ is
 included in a branch of $H$.  If it is $b \tp Q \tp c \tp d$, then
 $\{a, b\}$ is a cutset of $H$ contradictory to Lemma~\ref{l:2cutset}.
 So it is $b \tp a \tp P \tp d$, and $b \tp a \tp P \tp d \tp c$ is
 included in a branch of $H$.  Hence, $ab$, $cd$ are in the same
 branch of $H$, which contradicts our assumption.
\end{proof}

\begin{lemma}
  \label{l:cycleedgek4}
  Let $H$ be a subdivision of a $3$-connected graph.  Let $C$ be a
  cycle and $e$ an edge of $H$ such that $C$ and $e$ are edgewise
  disjoint.  Then there exists a subdivision of $K_4$ that is a
  subgraph of $H$ and that contains $C$ and $e$.
\end{lemma}

\begin{proof}
  Since $H$ is $2$-connected, there exist two vertex-disjoint paths
  $R=x \tp \cdots \tp x'$ and $S= y \tp \cdots \tp y'$ between $C$ and
  $e$, with $e=xy$ and $x', y'\in C$.  Let $P_1, P_2$ be the two
  edge-disjoint paths of $C$ with endvertices $x', y'$.  Let $P_3= x
  \tp \cdots \tp x' \tp y' \tp \cdots \tp y$.  Then $P_1, P_2, P_3$
  are three edge-disjoint paths between $x'$ and $y'$, so at most one
  of them is an edge.

  Vertices $x', y'$ have degree at least three in $H$, so they are
  also vertices of the $3$-connected graph of which $H$ is a
  subdivision.  So $H\setminus \{x', y'\}$ is connected.  Let $P$ be a
  shortest path connecting two paths among $P_1\setminus \{x', y'\},
  P_2\setminus \{x', y'\}, P_3\setminus \{x', y'\}$.  Then $P_1\cup
  P_2\cup P_3\cup P$ is a subdivision of $K_4$ satisfying the lemma.
\end{proof}


\section{Line graph of substantial  graphs}

A \emph{flat branch} in a graph is a branch such that no triangle
contains two vertices of it.  So a non-flat branch is an edge that
lies in a triangle.  Note that any branch of length zero is flat.
Moreover, every branch of length at least two is flat.

A \emph{triangular} subdivision of $K_4$ is a subdivision of $K_4$
that contains a triangle.  A \emph{square theta} is a theta that
contains a square, in other words, a theta with two branches of length
two.  A \emph{square prism} is a prism that contains a square, in
other words, a prism with two flat branches of length one.  Note that
a square prism is the line graph of a square theta.  A \emph{square
  subdivision of $K_4$} is a subdivision of $K_4$ whose corners form a
(possibly non-induced) square.  An induced square in a graph is
\emph{even} if an even number of edges of the square lie in a triangle
of the graph.  It easily checked that the line graph of a subdivision
$H$ of $K_4$ contains an even square if and only if $H$ is a square
subdivision of $K_4$; in that case the vertices in any even square of
$L(H)$ arise from the edges of a square on the branch-vertices of $H$.
It is easily checked that a prism contains only even squares.

A \emph{diamond} is a $K_4$ minus one edge.  A \emph{closed diamond}
is any graph obtained from a $K_4$ by subdividing only one edge.  In a
closed diamond that is not a $K_4$, the four corners induce a diamond,
there is a unique branch $P$ of length at least two, and we say that
$P$ \emph{closes} the diamond.

If $X, Y$ are two paths in a graph $G$, a \emph{connection between
$X$, $Y$} is a path $P = p \tp \cdots \tp p'$ such that $p$ has a
neighbor in $X$, $p'$ has a neighbor in $Y$, no interior vertex of $P$
has a neighbor in $X\cup Y$, and if $p\neq p'$, then $p$ has no
neighbor in $Y$ and $p'$ has no neighbor in $X$.

The line graph of $K_4$ is isomorphic to $K_{2, 2, 2}$ and is usually
called the \emph{octahedron}.  It has three even squares.  For every
even square $S$ of an octahedron $G$, the two vertices of $G\sm S$
are both links of $S$.  Note also that when $K$ is a square
prism with a square $S$, then $V(K) \setminus S$ is a link of
$S$.

\ 

Given a graph $G$, a graph $H$ such that $L(H)$ is an induced subgraph
of $G$, and a connected induced subgraph $C$ of $V(G) \setminus L(H)$,
we define several types that $C$ may have, according to its attachment
over $L(H)$:

\begin{itemize}
\item
  $C$ is of type \emph{branch} if the attachment of $C$ over $L(H)$
  is included in a flat branch of $L(H)$;

\item
  $C$ is of type \emph{triangle} if the attachment of $C$ over $L(H)$ is
  included in a triangle of $L(H)$;

\item $C$ is of type \emph{augmenting} if $C$ contains a connection
  $P= p \tp \cdots \tp p'$ between two distinct flat branches $X, Y$
  of $L(H)$ such that $N_X(p)$ is an edge of $X$, $N_Y(p')$ is an edge
  of $Y$, and there is no edge between $L(H)\setminus (X\cup Y)$ and
  $P$.  We say that $P$ is an \emph{augmenting path} for $C$.
\item $C$ is of type \emph{square} if $L(H)$ contains an even square
  $S$, $C$ contains a link $P$ of $S$, and there is no edge between
  $L(H)\setminus S$ and $P$.  We say that $P$ is a \emph{linking path}
  for $C$.
\end{itemize}

Note that the types may overlap: a subgraph $C$ may be of more than
one type.  Since we view a vertex of $G$ as a connected induced
subgraph of $G$, we may speak about the type of a vertex with respect
to $L(H)$.

\begin{lemma}
  \label{l:compprism}
  Let $G$ be a graph that contains no triangular ISK4. Let $K$ be a
  prism that is an induced subgraph of $G$ and let $C$ be a connected
  induced subgraph of $G\setminus K$.  Then $C$ is either of type
  branch,  triangle,  augmenting or square with respect to $K$.
\end{lemma}

\begin{proof}
  Let $X = x \tp \cdots \tp x'$, $Y = y \tp \cdots \tp y'$, $Z = z \tp
  \cdots \tp z'$ be the three flat branches of $K$ denoted in such a
  way that $xyz$ and $x'y'z'$ are triangles.  Call $X, Y, Z$ and the
  two triangles of $K$ the \emph{pieces} of $K$.  Suppose that $C$ is
  not of type branch or triangle and consider an induced subgraph $P$
  of $C$ minimal with respect to the property of being a connected
  induced subgraph, not of type branch or triangle.

  \begin{claim}
    \label{c:prismlocal}
    $P$ is a path, no internal vertex of $P$ has neighbors in $K$, and
    $N_K(P)$ is not included in a branch or triangle of $K$.
  \end{claim}

  \begin{proofclaim}
    If $P$ is not a path, then either $P$ contains a cycle or $P$ is a
    tree with a vertex of degree at least three.  In either case, $P$
    has three distinct vertices $a_1, a_2, a_3$ such that $P\sm a_i$
    is connected for each $i=1, 2, 3$ (if $P$ has a cycle, take any
    three vertices of $Z$; if $P$ is a tree, take three leaves of
    $P$).  For each $i=1, 2, 3$, by the minimality of $P$, the
    attachment of $P\sm a_i$ over $K$ is included in a piece $X_i$ of
    $K$, and $a_i$ has a neighbor $y_i$ in $V(K)\sm X_i$.  So we have
    $\{y_1, y_2\}\subseteq X_3$, $\{y_1, y_3\}\subseteq X_2$, $\{y_2,
    y_3\}\subseteq X_3$.  But this is impossible because no three
    pieces $X_1, X_2, X_3$ of $K$ have that property.  Thus $P$ is a
    path.  If $P$ has length zero, then the claim holds since, by the
    assumption, $P$ is not of type branch or triangle.  So, we may
    assume that $P$ has length at least one.  Let $P$ have ends $p,
    p'$.  Suppose that the claim fails.  Then by the minimality of
    $P$, we have $N_K (P\sm p') \subset A$ and $N_K(P\sm p) \subset
    B$, where $A, B$ are distinct pieces of $K$; moreover, some
    interior vertex of $P$ must have a neighbor in $K$.  So the
    attachment of the interior of $P$ over $K$ is not empty and is
    included in $A \cap B$.  Since two distinct flat branches of $K$
    are disjoint and two distinct triangles of $K$ are disjoint, we
    may assume that $N_K(p) \subseteq \{x, y, z\}$, $N_K(p') \subseteq
    X$, and some interior vertex of $P$ is adjacent to $x$.  Note that
    $p$ has at most two neighbors in $\{x, y, z\}$, because $G$ has no
    $K_4$, and that $p$ must have at least one neighbor in $\{y, z\}$,
    for otherwise $P$ is of type branch.  If $py, pz \in E(G)$, then,
    since some interior vertex of $P$ is adjacent to $x$, $P$ contains
    a path that closes the diamond $\{x, y, z, p\}$, a contradiction.
    So we may assume up to symmetry that $pz \in E(G)$ and $py \notin
    E(G)$.  Vertex $p'$ has a neighbor in $X\setminus x$, for
    otherwise $P$ is of type triangle.  Let $w$ be the neighbor of
    $p'$ closest to $x'$ along $X$.  Then $z \tp p \tp P \tp p' \tp w
    \tp X \tp x'$, $z \tp Z \tp z'$ and $z \tp y \tp Y \tp y'$ form a
    triangular ISK4, a contradiction.
  \end{proofclaim}

 Let $p, p'$ be the two ends of $P$.  We distinguish between two
 cases.
  
 \noindent{\it Case 1: $P$ is a connection between two flat branches
   of $K$ and has no neighbor in the third flat branch.} We may assume
 that $p$ has a neighbor in $X$, $p'$ has a neighbor in in $Y$, and
 none of $p, p'$ has neighbors in $Z$.  Let $x^L$ (resp.~$x^R$) be the
 neighbor of $p$ closest to $x$ (resp.~to $x'$) along $X$.  Let $y^L$
 (resp.~$y^R$) be the neighbor of $p'$ closest to $y$ (resp.~to $y'$)
 along $Y$.  If both $x^Lx^R, y^Ly^R$ are edges, then $C$ is of type
 augmenting and the lemma holds.  So let us assume up to symmetry that
 $x^Lx^R \notin E(G)$.  Suppose that $x^L \neq x^R$.  We may assume
 $y^L \neq y'$ (else $y^R\neq y$ and the argument is similar).  Then
 $p \tp x^L \tp X \tp x$, $p \tp x^R \tp X \tp x' \tp z' \tp Z \tp z$,
 $p \tp P \tp p' \tp y^L \tp Y \tp y$ form a triangular ISK4, a
 contradiction.  So $x^L = x^R$.  If $y^Ly^R$ is an edge, then $X\cup
 Y\cup P$ is a triangular ISK4.  So $y^Ly^R \notin E(G)$, and
 consequently $y^L = y^R$ (just like we obtained $x^L = x^R$).
 Suppose that $x^L$ is not equal to $x$ or $x'$.  We may assume that
 $y^L \neq y'$ (else $y^R\neq y$ and the argument is similar).  Then
 $x^L \tp X \tp x$, $x^L \tp p \tp P \tp p' \tp y^L \tp Y \tp y$ and
 $x^L \tp X \tp x' \tp z' \tp Z \tp z$ form a triangular ISK4, a
 contradiction.  So $x^L$ is one of $x, x'$ , and, similarly, $y^L$ is
 one $y, y'$.  We may assume $x^L = x$ and $y^L = y'$, for otherwise
 (\ref{c:prismlocal}) is contradicted.  Then $x \tp X \tp x'$, $x \tp
 p \tp P \tp p' \tp y'$, $x \tp z \tp Z \tp z'$ form a triangular
 ISK4, a contradiction.

  \noindent{\it Case 2: We are not in Case~1.} Suppose first that one
  of $p, p'$ has at least two neighbors in a triangle of $K$.  Then we
  may assume up to symmetry that $px, py \in E(G)$, and $pz \notin
  E(G)$ because $G$ contains no $K_4$.  By~(\ref{c:prismlocal}) and up
  to symmetry, $p'$ must have a neighbor in $Y\setminus y$ or in $Z$.
  Note that either $p=p'$ or $N_K(p) = \{x, y\}$, for otherwise $p$
  would contradict the minimality of $P$.  If $p'$ has a neighbor in
  $Z$, then let $w$ be such a neighbor closest to $z$ along $Z$.  Then
  $p \tp P \tp p' \tp w \tp Z \tp z$ closes the diamond $\{p, x, y,
  z\}$, a contradiction.  So, $p'$ has no neighbor in $Z$, and so it
  has neighbors in $Y\setminus y$.  Let $w^L$ (resp.~$w^R$) be the
  neighbor of $p'$ closest to $y$ (resp.~to $y'$) along $Y$.  Note
  that $w^R \neq y$ by~(\ref{c:prismlocal}).  If $p'$ has no neighbor
  in $X$, then $x \tp X \tp x'$, $x \tp p \tp P \tp p' \tp w^R \tp Y
  \tp y'$ and $x \tp z \tp Z \tp z'$ form a a triangular ISK4.  So
  $p'$ has a neighbor in $X$, and we denote by $v^L$ (resp.~$v^R$)
  such a neighbor closest to $x$ (resp.~to $x'$) along $X$.  Since we
  are not in Case~1, we have $p \neq p'$.  If either $v^L\neq x'$ or
  $w^L\neq y'$, then $p'$ contradicts the minimality of $P$.  So
  assume $v^L = v^R = x'$ and $w^L = w^R = y'$.  If $X$ has length at
  least two, then $p \tp P \tp p' \tp x' \tp z' \tp Z \tp z$ closes
  the diamond $\{p, x, y, z\}$.  So $X$ has length one, and similarly
  $Y$ has length one.  But then $P$ is a link of the even
  square $\{x, y, x', y'\}$ of $K$, so $C$ is of type square.

  Now we assume that both $p, p'$ have at most one neighbor in a
  triangle of $K$.  At least one of $p, p'$ (say $p$) must have
  neighbors in more than one branch of $K$, for otherwise we are in
  Case~1.  So $p = p'$ by the minimality of $P$, and $p$ has neighbors
  in $X, Y, Z$, for otherwise we are again in Case~1.  We may assume
  that $py, pz \notin E(G)$.  Let $x^R$, $y^R$, $z^R$ be the neighbors
  of $p$ closest to $x', y', z'$ along $X, Y, Z$ respectively.  Then
  $p\tp x^R \tp X \tp x'$, $p \tp y^R \tp Y \tp y'$, $p \tp z^R \tp Z
  \tp z'$ form a triangular ISK4, a contradiction.
\end{proof}

\begin{lemma}
  \label{l:compsk4}
  Let $G$ be a graph that contains no triangular ISK4.  Let $H$ be a
  subdivision of $K_4$ such that $L(H)$ is an induced subgraph of $G$.
  Let $C$ be a connected induced subgraph of $G\setminus L(H)$.  Then
  $C$ is either of type branch, triangle, augmenting or square with
  respect to $L(H)$.
\end{lemma}

\begin{figure}[h]
  \center
  \includegraphics{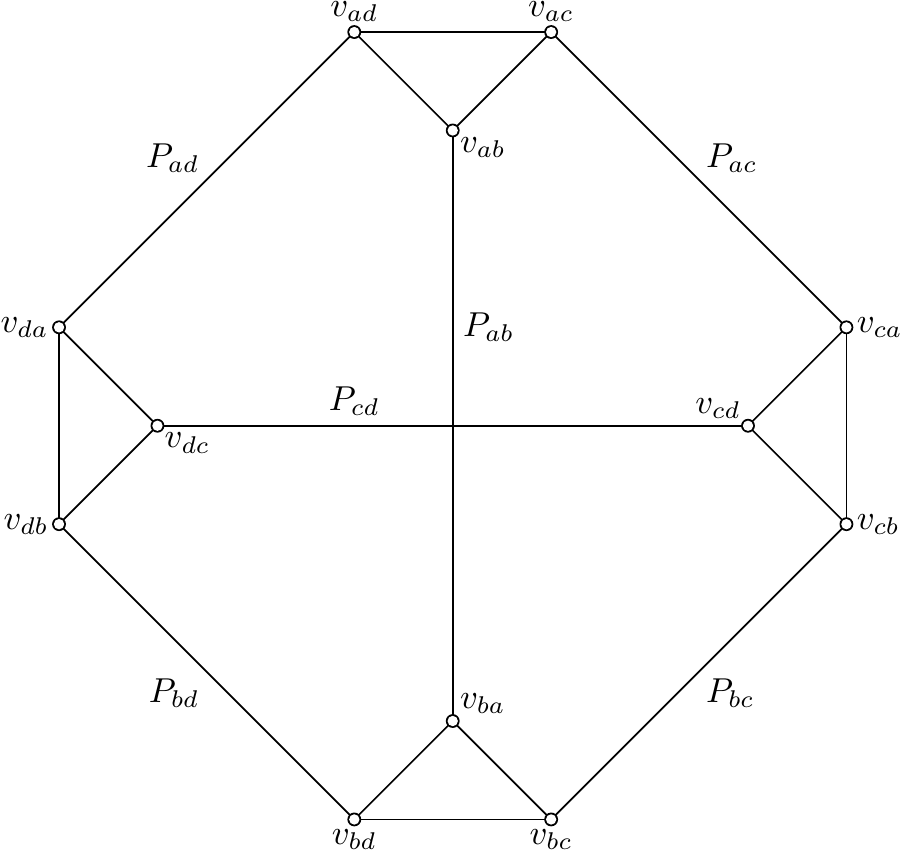}
  \caption{The line graph of a subdivision of $K_4$\label{fig:lsk4}}
\end{figure}

\begin{proof}
  Let $a, b, c, d$ be the four corners of $H$.  See
  Figure~\ref{fig:lsk4}.  The three edges incident to each vertex $x =
  a, b, c, d$ form a triangle in $L(H)$, which we label $T_x$.  In
  $L(H)$, for every pair $x, y \in \{a, b, c, d\}$ there is one path
  with an end in $T_x$ and an end in $T_y$, and no interior vertex in
  the triangles, and we denote this path by $P_{xy}$.  Note that
  $P_{xy}= P_{yx}$, and the six distinct paths so obtained are vertex
  disjoint.  Some of these paths may have length $0$.  In the triangle
  $T_x$, we denote by $v_{xy}$ the vertex that is the end of the path
  $P_{xy}$.  Thus the flat branches of $L(H)$ are the paths of length
  at least one among $P_{ab}$, $P_{ac}$, $P_{ad}$, $P_{bc}$, $P_{bd}$,
  $P_{cd}$.  Note that $L(H)$ may have as many as four triangles other
  than $T_a, T_b, T_c, T_d$.  The branch-vertices of $L(H)$ are
  $v_{ab}$, $v_{ac}$, $v_{ad}$, $v_{ba}$, $v_{bc}$, $v_{bd}$,
  $v_{ca}$, $v_{cb}$, $v_{cd}$, $v_{da}$, $v_{db}$ and $v_{dc}$.  The
  subgraph $L(H)$ has no other edges than those in the four triangles
  and those in the six paths.  Let every flat branch and every
  triangle of $L(H)$ be called a \emph{piece} of $L(H)$.

  Suppose that $C$ is not of type branch or triangle with respect to
  $L(H)$, and consider an induced subgraph $P$ of $C$ minimal with
  respect to the property of being a connected induced subgraph not of
  type branch or triangle.

  \begin{claim}
    \label{c:local}
    $P$ is a path, no internal vertex of $P$ has neighbors in $L(H)$
    and $N_{L(H)}(P)$ is not included in a flat branch or in a
    triangle of~$L(H)$.
  \end{claim}

  \begin{proofclaim}
    If $P$ is not a path, then, as in the proof of
    Claim~\ref{c:prismlocal} in Lemma~\ref{l:compprism}, $P$ has three
    distinct vertices $a_1, a_2, a_3$ such that $P\sm a_i$ is
    connected for each $i=1, 2, 3$.  For each $i=1, 2, 3$, by the
    minimality of $P$, the attachment of $P\sm a_i$ over $K$ is
    included in a piece $X_i$ of $K$, and $a_i$ has a neighbor $y_i$
    in $V(K)\sm X_i$.  So we have $\{y_1, y_2\}\subseteq X_3$, $\{y_1,
    y_3\}\subseteq X_2$, $\{y_2, y_3\}\subseteq X_3$.  This is
    possible in $L(H)$ only if each of $X_1, X_2, X_3$ is a triangle
    and $\{y_1, y_2, y_3\}$ is also a triangle.  But then the
    attachment of $P$ is $\{y_1, y_2, y_3\}$, so $P$ is of type
    triangle, a contradiction.  So $P$ is a path.  If $P$ has length
    zero, then the claim holds since, by the assumption, $P$ is not of
    type branch or triangle.  So, we may assume that $P$ has length at
    least one.  Let $P$ have ends $p, p'$.  Suppose that the claim
    fails.  Then by the minimality of $P$, $ N_{L(H)} (p) \subset A$
    and $N_{L(H)}(p') \subset B$, where $A, B$ are distinct pieces of
    $L(H)$.  Also some interior vertex of $P$ must have a neighbor in
    $L(H)$.  By the minimality of $P$, the attachment of the interior
    of $P$ over $L(H)$ is included in $A \cap B$.  Since two distinct
    flat branches of $L(H)$ are disjoint, we may assume that $A= T_d$
    and either $B = P_{ad}$ or $P_{ad}$ has length zero and $B = T_a$.
    In either case, $A\cap B = \{v_{da}\}$.  Note that $p$ has at most
    two neighbors in $T_d$, because $G$ has no $K_4$, and that $p$
    must have at least one neighbor in $\{v_{db}, v_{dc}\}$, for
    otherwise the attachment of $P$ is included in $B$ and $P$ is of
    type branch or triangle.  Note that $p'$ has neighbors in
    $B\setminus v_{da}$, for otherwise $P$ is of type triangle.  If
    $pv_{db}, pv_{dc} \in E(G)$, then since some interior vertex of
    $P$ is adjacent to $v_{da}$, $P$ contains a subpath that closes
    the diamond $T_d \cup \{ p\}$, a contradiction.  So, up to
    symmetry, we assume $pv_{db} \in E(G)$ and $pv_{dc} \notin E(G)$.
 
    We observe that $P \cup P_{ac} \cup B$ contains an induced path
    $Q$ from $p$ to $v_{ca}$, and no vertex of $Q$ has neighbors in
    $V(P_{cd}) \cup V(P_{bd}) \cup V(P_{bc})$.  If possible, choose
    $Q$ so that it does not contain $v_{ab}$.  Now $Q$, $P_{cd}$,
    $P_{bd}$, $P_{bc}$, form a triangular ISK4 (whose triangle is
    $T_c$ and fourth corner is $v_{db}$) except if $Q$ goes through
    $v_{ab}$ and $P_{ab}$ has length zero (so $v_{ab} = v_{ba}$).  In
    the latter situation, we must have $N_B(p') = \{v_{ab}\}$ by the
    choice of $Q$, so $B= T_a$ and $P_{ad}$ has length zero.  If
    $P_{bd}$ has length at least $1$, then $v_{db} \tp P \tp p' \tp
    v_{ba}$, $v_{db} \tp P_{bd} \tp v_{bd}$ and $v_{db} \tp v_{dc} \tp
    P_{cd} \tp v_{cd} \tp v_{cb} \tp P_{bc} \tp v_{bc}$ form a
    triangular ISK4.  So $P_{bd}$ has length zero.  But then
    $\{v_{da}, v_{db}, v_{ab}\}$ is a triangle and is the attachment
    of $P$ over $L(H)$, so $P$ is of type triangle with respect to
    $L(H)$, a contradiction.
  \end{proofclaim}

  \begin{claim}
    \label{c:octah}
    One of $P_{ab}$, $P_{ac}$, $P_{ad}$, $P_{bc}$, $P_{bd}$, $P_{cd}$
    has length at least $1$. 
  \end{claim} 
  
  \begin{proofclaim}
   Suppose that $P_{ab}$, $P_{ac}$, $P_{ad}$, $P_{bc}$, $P_{bd}$,
   $P_{cd}$ all have length zero.  Then $L(H)$ is the octahedron
   ($K_{2, 2, 2}$).  Note that $L(H)$ has no flat branch.  For
   convenience, we rename its vertices $x$, $x'$, $y$, $y'$, $z$, $z'$
   so that $xx', yy', zz' \notin E(L(H))$ and all other pairs of
   distinct vertices are edges.  If $P$ has at most one neighbor in
   every pair $\{x, x'\}$, $\{y, y'\}$, $\{z, z'\}$, then
   $N_{L(H)}(P)$ is a subset of a triangle, a contradiction.  So, we
   may assume up to symmetry that $p$ is adjacent to $x$ and $p'$ to
   $x'$.  Let $S$ be the square of $L(H)$ induced by $y, y', z, z'$.
   Vertex $p$ cannot be adjacent to the two vertices of an edge of
   $S$, for that would yield (with $x$) a $K_4$ in $G$.  So we may
   assume $py, py' \notin E(G)$.  If $pz, pz'$ are both in $E(G)$,
   then $p$ itself is a vertex not of type branch or triangle, so
   $p=p'$ by the minimality of $P$, and since $S' = \{x, x', z, z'\}$
   is an even square of $L(H)$ and $N_{L(H)}(P) = S'$, $C$ is of type
   square.  Hence we may assume up to symmetry that $pz'\notin E(G)$,
   so $p$ has at most one neighbor in $S$.  Similarly, $p'$ has at
   most one neighbor in $S$.  If any edge $uv$ of $S$ has no neighbor
   of $p$ or $p'$, then $P$ closes the diamond induced by $\{u, v, x,
   x'\}$, a contradiction.  So every edge of $S$ has a neighbor of $p$
   or $p'$, which implies $pz\in E(G)$ and $p'z'\in E(G)$.  Then $P$
   is a link of the square $\{x, z, x', z'\}$ of $L(H)$, so $C$
   is of type square.
  \end{proofclaim}

  By (\ref{c:octah}) we may assume up to symmetry that $P_{ab}$ has
  length at least one.  So the vertices of $P_{ad}, P_{bd}, P_{ab},
  P_{ac}, P_{bc}$ induce a prism $K$ in $G$, whose triangles are $T_a,
  T_b$ and whose flat branches are $P_{ab}, P_{ac}\cup P_{bc}$ and
  $P_{ad}\cup P_{bd}$.  We apply Lemma~\ref{l:compprism} to $K$ and
  $P$, which leads to the following four cases.
  
  \medskip

  \noindent{\it Case 1: $P$ is of type branch with respect to $K$.}
  Suppose first that $N_K(P) \subseteq V(P_{ab})$.
  By~(\ref{c:local}), $P$ has neighbors in $P_{cd}$, and we may assume
  that $p$ has a neighbor in $P_{ab}$, $p'$ has a neighbor in
  $P_{cd}$, and no proper subpath of $P$ has this property.  Let $v^L$
  (resp.~$v^R$) be the neighbor of $p$ closest to $v_{ab}$ (resp.~to
  $v_{ba}$) along $P_{ab}$.  Up to the symmetry between $P_{ab}$ and
  $P_{cd}$ we may assume $v^Lv^R \notin E(G)$, for otherwise $C$ is of
  type augmenting with respect to $L(H)$ and the lemma holds.  Let
  $w^R$ the neighbor of $p'$ closest to $v_{cd}$ along $P_{cd}$.  If
  $v^L = v^R$, then $v^L \tp P_{ab} \tp v_{ab} \tp v_{ac} \tp P_{ac}
  \tp v_{ca}$, $v^L \tp P_{ab} \tp v_{ba} \tp v_{bc} \tp P_{bc} \tp
  v_{cb}$, $v^L \tp p \tp P \tp p' \tp w^R \tp P_{cd} \tp v_{cd}$ form
  a triangular ISK4, a contradiction.  If $v^L \neq v^R$, then $p \tp
  v^L \tp P_{ab} \tp v_{ab} \tp v_{ac} \tp P_{ac} \tp v_{ca}$, $p \tp
  v^R \tp P_{ab} \tp v_{ba} \tp v_{bc} \tp P_{bc} \tp v_{cb}$, $p \tp
  P \tp p' \tp w^R \tp P_{cd} \tp v_{cd}$ form a triangular ISK4, a
  contradiction.

  Now we may assume up to symmetry that $N_{K}(P) \subseteq V(P_{ad})
  \cup V(P_{bd})$.  Suppose that $P$ has a neighbor in each of
  $P_{ad}$, $P_{bd}$ and $P_{cd}$.  Let $v^a, v^b, v^c$ be the
  neighbors of $P$ closest to $v_{da}$, $v_{db}$ and $v_{dc}$
  respectively along these paths.  Then $V(P) \cup V(v^a \tp P_{ad}
  \tp v_{da}) \cup V(v^b \tp P_{bd} \tp v_{db}) \cup V(v^c \tp P_{cd}
  \tp v_{dc})$ induces a triangular ISK4 (whose corners are $v_{da},
  v_{db}, v_{dc}$ and one of $p,p'$), a contradiction.  So, $P$ has no
  neighbor in at least one of $P_{ad}$, $P_{bd}$, $P_{cd}$.

  If $P$ has no neighbor in $P_{bd}$, then by~(\ref{c:local}), we may
  assume that $p$ has a neighbor in $P_{ad}$, $p'$ has a neighbor in
  $P_{cd}$, and no proper subpath of $P$ has such a property.  Let
  $v^R$ be the neighbor of $p$ closest to $v_{ad}$ along $P_{ad}$.
  Suppose that $p'$ has a unique neighbor $w$ in $P_{cd}$.  If $v^R =
  v_{da}$, then $w \neq v_{dc}$ by~(\ref{c:local}) and $w \tp P_{cd}
  \tp v_{dc}$, $w \tp p' \tp P \tp p \tp v_{da}$, $w \tp P_{cd} \tp
  v_{cd} \tp v_{cb} \tp P_{bc} \tp v_{bc} \tp v_{bd} \tp P_{bd} \tp
  v_{db}$ form a triangular ISK4.  If $v^R \neq v_{da}$, then $w \tp
  p' \tp P \tp p \tp v^R \tp P_{ad} \tp v_{ad}$, $w \tp P_{cd} \tp
  v_{cd} \tp v_{ca} \tp P_{ac} \tp v_{ac}$, $w \tp P_{cd} \tp v_{dc}
  \tp v_{db} \tp P_{bd} \tp v_{bd} \tp v_{ba} \tp P_{ab} \tp v_{ab}$
  form a triangular ISK4, a contradiction.  So $p'$ has at least two
  neighbors on $P_{cd}$, and in particular $P_{cd}$ has length at
  least one.  So $P_{cd}, P_{ad}, P_{ac}, P_{bd}, P_{cb}$ form a prism
  $K'$.  Let us apply Lemma~\ref{l:compprism} to $K'$ and $P$.  Since
  $P$ has at least two neighbors in the flat branch $P_{cd}$ of $K'$
  and at least one neighbor in $P_{ad}$, $P$ is not of type branch or
  triangle with respect to $K'$.  Also $P$ is not of type square with
  respect to $K'$, because $N_{K'}(P)$ is included in $V(P_{ad}) \cup
  V(P_{cd})$ and cannot induce an even square of $K'$.  So $P$ is of
  type augmenting with respect to $K'$.  So $N_{K'}(p)$ is an edge of
  $P_{ad}$ (and this implies that $P_{ad}$ is a flat branch of
  $L(H)$), $N_{K'}(p')$ is an edge of $P_{cd}$, hence $P$ is of type
  augmenting with respect to $L(H)$.   
  
  If $P$ has no neighbor in $P_{ad}$, the situation is similar to the
  preceding paragraph (by symmetry).
  
  Now suppose that $P$ has no neighbor in $P_{cd}$.
  By~(\ref{c:local}), we may assume that $p$ has a neighbor in
  $P_{ad}$, $p'$ has a neighbor in $P_{bd}$, and no proper subpath of
  $P$ has this property.  Let $v^R$ (resp.~$v^L$) be the neighbor of
  $p$ closest to $v_{ad}$ (resp.~to $v_{da}$) along $P_{ad}$.  Let
  $w^R$ (resp.~$w^L$) be the neighbor of $p'$ closest to $v_{db}$
  (resp.~to $v_{bd}$) along $P_{bd}$.  If both $v^Lv^R, w^Lw^R$ are
  edges, then $C$ is of type augmenting with respect to $L(H)$ and the
  lemma holds.  So let us assume, up to the symmetry between $P_{ad}$
  and $P_{bd}$, that $v^Lv^R$ is not an edge.  If $v^L \neq v^R$, then
  $p \tp v^L \tp P_{ad} \tp v_{da}$, $p \tp v^R \tp P_{ad} \tp v_{ad}
  \tp v_{ac} \tp P_{ac} \tp v_{ca} \tp v_{cd} \tp P_{cd} \tp v_{dc}$
  and $p \tp P \tp p' \tp w \tp P_{bd} \tp v_{db}$ form a triangular
  ISK4, a contradiction.  So $v^L = v^R$.  If $w^Rw^L$ is an edge,
  then $P_{ab}\cup P_{ad}\cup P_{bd}\cup P$ is a triangular ISK4, a
  contradiction. So $w^Rw^L$ is not an edge, and, as above, this
  implies that $w^R=w^L$.  We cannot have $\{v^L, w^L\}= \{v_{da},
  v_{db}\}$, for otherwise $N_{L(H)}(P) \subseteq T_d$, contradictory
  to~(\ref{c:local}).  So we may assume that $v ^L\neq v_{da}$.  Then
  $v^L \tp P_{ad} \tp v_{da}$, $v^L \tp P_{ad} \tp v_{ad} \tp v_{ac}
  \tp P_{ac} \tp v_{ca} \tp v_{cd} \tp P_{cd} \tp v_{dc}$ and $v^L \tp
  p \tp P \tp p' \tp w \tp P_{bd} \tp v_{db}$ form a triangular ISK4,
  a contradiction.

\medskip

  \noindent{\it Case 2: $P$ is of type triangle with respect to $K$.}
  We assume up to symmetry that $N_K(P) \subseteq T_a$.
  By~(\ref{c:local}) and up to symmetry, we may assume that $p$ has a
  neighbor in $T_a$, $p'$ has a neighbor in $P_{cd}$, and no interior
  vertex of $P$ has a neighbor in $L(H)$.  We may assume that we are
  not in Case~1, so $p$ has at least two neighbors in $T_a$; and $p$
  has only two neighbors in $T_a$, for otherwise there is a $K_4$ in
  $G$.  Suppose that $pv_{ac}, pv_{ad}\in E(G)$ and $pv_{ab} \notin
  E(G)$.  If $p'$ has only one neighbor in $P_{cd}$, then $P_{ac}\cup
  P_{ad}\cup P_{cd}\cup P$ is a triangular ISK4, a contradiction.  So
  $p'$ has at least two neighbors in $P_{cd}$, which implies that
  $P_{cd}$ has length at least one, and we may assume up to symmetry
  that the neighbor $w$ of $p'$ closest to $d$ on $P_{cd}$ is
  different from $c$.  Then $v_{ac} \tp P_{ac} \tp v_{ca} \tp v_{cb}
  \tp P_{cb} \tp v_{bc}$, $v_{ac} \tp v_{ab} \tp P_{ab} \tp v_{ba}$
  and $v_{ac} \tp p\tp P\tp p' \tp w \tp P_{dc} \tp v_{dc} \tp v_{db}
  \tp P_{db} \tp v_{bd}$ form a triangular ISK4 (whose corners are the
  vertices of $T_b$ and $v_{ac}$), a contradiction.  So $pv_{ab} \in
  E(G)$ and we may assume up to symmetry $pv_{ad} \notin E(G)$.  Then
  $v_{ab} \tp p \tp P \tp p' \tp w \tp P_{cd} \tp v_{dc}$, $ v_{ab}
  \tp v_{ad} \tp P_{ad} \tp v_{da}$ and $ v_{ab} \tp P_{ab} \tp v_{ba}
  \tp v_{bd}\tp P_{bd} \tp v_{db}$ form an triangular ISK4, a
  contradiction.
  
  \medskip

  \noindent{\it Case 3: $P$ is of type augmenting with respect to
  $K$.}  We may assume up to symmetry that $N_K(p)$
  is an edge $e$ in $P_{ad} \cup P_{bd}$ and $N_K(p')$ is an edge $e'$
  in either $P_{ab}$ or in $Q = P_{ac} \cup P_{bc}$.  If $e'$ is in
  $P_{ab}$, let $v^R$ be its vertex closest to $v_{ba}$.  If $e'$ is
  in $Q$ let $v^R$ be its vertex closest to $v_{bc}$.  Let $u^R$ be
  the other vertex of $e'$.

  Suppose that $e = v_{da}v_{db}$.  So $T_d\cup\{p\}$ induces a
  diamond.  Then $P$ has no neighbor in $P_{cd}$, for otherwise $P
  \cup P_{cd}$ would contain a path that closes the diamond
  $T_d\cup\{p\}$.  If $e'$ is in $P_{ab}$, then $v_{da} \tp p \tp P
  \tp p' \tp v^R \tp P_{ab} \tp v_{ba} \tp v_{bc} \tp P_{bc} \tp
  v_{cb}$, $v_{da} \tp v_{dc} \tp P_{cd} \tp v_{cd}$ and $v_{da} \tp
  P_{ad} \tp v_{ad} \tp v_{ac} \tp P_{ac} \tp v_{ca}$ form a
  triangular ISK4, a contradiction (note that this holds even when $P$
  and every $P_{xy}$ except $P_{ab}$ has length zero).  Hence $e'$ is
  in $Q$.  If $v^R$ is in $P_{ac}$, then $P_{ac}$ has length at least
  one and $v^R \neq v_{ac}$, so $p \tp P \tp p' \tp v^R \tp P_{ac} \tp
  v_{ca} \tp v_{cd} \tp P_{cd} \tp v_{dc}$ closes the diamond $\{p,
  v_{da}, v_{db}, v_{dc}\}$.  So $v^R$ is not in $P_{ac}$; and, by
  symmetry, $u^R$ is not in $P_{bc}$, so we must have $e' =
  v_{ca}v_{cb}$.  If one of $P_{ac}, P_{ad}$ has length at least one,
  then $p \tp P \tp p' \tp v_{ca} \tp v_{cd} \tp P_{cd} \tp v_{dc}$
  closes the diamond $T_d\cup\{p\}$, a contradiction.  So suppose that
  both $P_{ad}$, $P_{ac}$ have length zero, and similarly both
  $P_{bd}$, $P_{bc}$ have length zero.  Then $P$ is a link of
  the even square induced by the four vertices $v_{da} = v_{ad}$,
  $v_{ac} = v_{ca}$, $v_{cb} = v_{bc}$ and $v_{bd} = v_{db}$ of
  $L(H)$, hence, $C$ is of type square with respect to $L(H)$.

  Now we may assume that $e \neq v_{da} v_{db}$, and, similarly, that
  $e' \neq v_{ca} v_{cb}$.  We may assume up to symmetry that $e$ is
  in $P_{ad}$.  We know that $e'$ is in either $P_{ab}$, $P_{ac}$ or
  $P_{bc}$, and that no vertex of $P$ has a neighbor in $P_{bd}$.  Let
  $e=u^Lv^L$ so that the vertices $v_{ad}, u^L, v^L, v_{da}$ lie in
  this order on $P_{ad}$.  Suppose that some vertex of $P_{cd}$ has a
  neighbor in $P$ and call $w$ such a vertex closest to $v_{dc}$.
  Note that $w$ must be adjacent to $x \in \{p, p'\}$, so $x$ itself
  is a connected induced subgraph of $G$, not of type branch or
  triangle with respect to $L(H)$.  This and the minimality of $P$
  imply $x=p=p'$.  Put $Q_1 = p \tp v^L \tp P_{ad} \tp v_{da}$, $Q_2 =
  p \tp w \tp P_{cd} \tp v_{dc}$.  If $e'$ is in $P_{ab}$, put $Q_3 =
  p \tp v^R \tp P_{ab} \tp v_{ba} \tp v_{bd} \tp P_{bd} \tp v_{db}$.
  If $e'$ is in $Q$, put $Q_3 = p \tp v^R \tp Q \tp v_{bc} \tp v_{bd}
  \tp P_{bd} \tp v_{db}$.  Now, if $w$ has no neighbor in $Q_3$, then
  $Q_1, Q_2, Q_3$ form a triangular ISK4, a contradiction.  So $w$ has
  a neighbor in $Q_3$, which means that $w=v_{cd}$ and $v^R \in
  P_{ac}$.  Then $p \tp v_{cd} \tp v_{cb} \tp P_{bc} \tp v_{bc} \tp
  v_{ba} \tp P_{ab} \tp v_{ab}$, $p \tp u^R \tp P_{ac} \tp v_{ac}$ and
  $p \tp u^L \tp P_{ad} \tp v_{ad}$ form a triangular ISK4, a
  contradiction.  So no vertex of $P$ has a neighbor in $P_{cd}$.  It
  follows that $C$ is of type augmenting with respect to $L(H)$.
  
  \medskip

  \noindent{\it Case 4: $P$ is of type square with respect to $K$.}
  So $P$ is a link of an even square $S$ of $K$ and has no neighbor in
  $K\setminus S$.  We may assume up to symmetry that $S$ contains
  $P_{ad}$ and $P_{bd}$, so these two paths have length zero, that is,
  $v_{ad}= v_{da}$ and $v_{bd}= v_{db}$.  If any vertex of $P$ has a
  neighbor $w$ in $P_{cd}$, then $p=p'$ by the minimality of $P$.  So
  $p$ is adjacent to both $v_{ad}, v_{bd}$.  Then $T_d\cup\{p\}$
  induces either a $K_4$ (if $w = v_{dc}$) or a diamond that is closed
  by a subpath of $P_{cd}\cup\{p\}$, a contradiction.  Hence, no
  vertex of $P$ has a neighbor in $P_{cd}$.  Suppose that $P_{ab}
  \subset S$.  Note that $S$ is an even square of $K$, but a non-even
  square of $L(H)$.  Then $V(P) \cup \{v_{da}, v_{ba}\}$ contains an
  induced path $Q$ from $v_{da}$ to $v_{ba}$ such that no interior
  vertex of $Q$ has a neighbor in $(L(H)\setminus S) \cup \{v_{da},
  v_{ba}\}$.  Then $v_{da} \tp Q \tp v_{ba} \tp v_{bc} \tp P_{bc} \tp
  v_{cb}$, $v_{da} \tp v_{ac} \tp P_{ac} \tp v_{ca}$ and $v_{da} \tp
  v_{dc} \tp P_{cd} \tp v_{cd}$ form a triangular subdivision of
  $K_4$, a contradiction.  So $P_{ab} \not\subset S$.  So $S$ has
  vertices $v_{ad} = v_{da}$, $v_{db} = v_{bd}$, $v_{bc} = v_{cb}$ and
  $v_{ac} = v_{ca}$, and $S$ is an even square of $L(H)$.  Thus $C$ is
  of type square with respect to $L(H)$ because of $S$ and $P$.
\end{proof}

Let us say that a graph is \emph{substantial} if it is cyclically
$3$-connected and not a square theta or a square subdivision of $K_4$.
The following lemma shows that type square arises only with
line graphs of non substantial graphs.

\begin{lemma}
  \label{l:compsubstan}
  Let $G$ be a graph that contains no triangular ISK4.  Let $H$ be a
  substantial graph such that $L(H)$ is an induced subgraph of $G$.
  Let $C$ be a component of $G\setminus L(H)$.  Then $C$ is either of
  type branch, triangle or augmenting with respect to $L(H)$.
\end{lemma}

\begin{proof}
  We suppose that $C$ is minimal with respect to the property of being
  not of type branch or triangle with respect to $L(H)$.  Note that
  every vertex in $H$ has degree at most three since $L(H)$ contains
  no $K_4$.  We may assume that there are two non-incident edges $e_1,
  e_2$ of $H$ that are members of the attachment of $C$ over $L(H)$
  and are not in the same branch of $H$, for otherwise all edges of
  the attachment of $C$ over $L(H)$ are either in the same branch of
  $H$, and so $C$ is of type branch or triangle, or are pairwise
  incident, and so $C$ is of type triangle.  Since $H$ is
  $2$-connected, there exists a cycle $Z$ of $H$ that goes through
  $e_1, e_2$, and we put $e_1=ab$, $e_2=cd$ so that $a, b, c, d$
  appear in this order along $Z$.  Note that $a, b, c, d$ are pairwise
  distinct.  Let $P$ be the subpath of $Z$ from $a$ to $d$ that does
  not contain $b, c$, and let $Q$ be the subpath of $Z$ from $b$ to
  $c$ that does not contain $a, d$.  By Lemma~\ref{l:2edges} there is
  a path $R$ with an end-vertex in $P$, an end-vertex in $Q$ and no
  interior vertex in $C$, and $R$ is not from $a$ to $b$ or from $c$
  to $d$.

  Suppose that $V(H) = V(P) \cup V(Q) \cup V(R)$.  Then $R$ must have
  length at least two, and $H$ must be a theta since it is
  substantial, so $L(H)$ is a prism.  By the preceding paragraph, the
  attachment of $C$ over $L(H)$ contains at least two vertices in
  distinct flat branches $L(H)$, and not in a triangle of that prism.
  So, by Lemma~\ref{l:compprism}, $C$ is of type augmenting or square
  with respect to the prism.  Moreover, type square is impossible
  because $H$ is substantial; so $C$ is of type augmenting, and the
  lemma holds.

  Now we may assume that $H$ has more edges than those in $P, Q, R$.
  By Lemma~\ref{l:descc3c}, $H$ is a subdivision of a $3$-connected
  graph.  Pick any $r \in V(P)\cap V(R), r' \in V(Q) \cap V(R)$ and
  put $P_1 = r P a b Q r'$, $P_2 = r P d c Q r'$, and $P_3 = R = r \tp
  \cdots \tp r'$.  By Lemma~\ref{l:thetask4}, for some distinct
  $i,j\in\{1, 2, 3\}$ there exists a path $S$ of $H$ with an end in
  the interior of $P_i$, an end in the interior of $P_j$ and such that
  the interior of $S$ is disjoint from $P_1, P_2, P_3$.  Since $H' =
  P_1 \cup P_2 \cup P_3 \cup S$ is a subdivision of $K_4$, we may
  apply Lemma~\ref{l:compsk4} to $C$ and $L(H')$.  Note that $C$
  cannot be of type branch or triangle with respect to $L(H')$ because
  of the edges $ab$ and $cd$.  Hence $C$ is of type square or
  augmenting with respect to $L(H')$, and, by the minimality of $C$,
  it is either a link of an even square of $L(H')$ or a connection
  between two branches of $L(H')$.  We claim that the interior
  vertices of $C$ have no neighbor in $L(H')$.  For suppose on the
  contrary that there is a vertex $w$ of $L(H')$ with a neighbor in
  the interior of $C$.  If $C$ is of type augmenting with respect to
  $L(H')$, then, by the minimality of $C$, $w$ must lie in the
  intersection of two edges of distinct flat branches of $L(H')$, a
  contradiction since flat branches of $L(H')$ do not intersect.  If
  $C$ is of type square with respect to $L(H')$, then, by the
  minimality of $C$, $w$ must lie in the intersection of two triangles
  of $L(H')$ that share a common vertex not in the square.  But then
  $C$ contains a path that closes a diamond, a contradiction.  So the
  claim is proved.  Now, we distinguish between two cases.
  
  \medskip

  \noindent{\it Case 1: $H$ contains a square subdivision of $K_4$ as
  a subgraph, and $C$ is of type square with respect to its
  line graph.} We may assume up to a relabeling that $C$ is of type
  square with respect to $L(H')$ and that $abcd$ is a square of $H$,
  $P_1 = ab$, $P_2 = dc$, $R$ is from $a$ to $c$ and $S$ is from $b$
  to $d$.  Every vertex of $H$ has degree at most three since $L(H)$
  contains no $K_4$.  Since $H$ is substantial, it is not a square
  subdivision of $K_4$, so there is a vertex in $H\setminus H'$.
  Since $H$ is connected and $H \neq H'$, there exists a neighbor in
  $V(H) \setminus V(H')$ of a vertex $e \in V(H')$, and $e\notin \{a,
  b, c, d\}$ because $a, b, c, d$ have already three neighbors.  So
  $e$ is in the interior of one of $S, R$ (say $S$).  Since $H$ is
  $2$-connected, $\{e\}$ is not a cutset of $H$ and there exists a
  path $T$ from~$e$ to some other vertex in $H'$.  If every such path
  has its two ends in $S$, then we put $A = V(P) \cup V(Q) \cup V(R)$,
  $B = (V(H) \setminus A) \cup \{b, d\}$ and we observe that $(A, B)$
  is a cyclic $2$-separation of $H$, a contradiction.  So we may
  assume that the other end $e'$ of $T$ is in the interior of $R$.
  Now let $H''$ be the subgraph of $H$ obtained from $P \cup Q \cup R
  \cup S \cup T$ by deleting the edges of the subpath $d$-$S$-$e$.  We
  observe that $H''$ is a subdivision of $K_4$ (whose corners are $a,
  b, c, e'$).  We now apply Lemma~\ref{l:compsk4} to $C$ and $L(H'')$.
  In fact $C$ cannot be of type branch, triangle or augmenting with
  respect to $L(H'')$, because $C$ has a neighbor in three distinct
  branches of $L(H'')$; and $C$ cannot be of type square because the
  edges $ab, bc, cd, da$ of $H$ do not form an even square in $L(H'')$
  since $d$ has degree two in $H''$.  This is a contradiction.
  
    \medskip
  
    \noindent{\it Case 2: We are not in Case~1.}  So $C$ is of type
    augmenting with respect to $L(H')$.  We may assume, up to a
    relabeling, that the attachment of $C$ over $L(H')$ consists of
    two pairs $\{e_1, e'_1\}$, $\{e_2, e'_2\}$ of adjacent vertices,
    where (in $H$) $e_1, e'_1$ are two incident edges of $P_1$ and
    $e_2, e'_2$ are two incident edges of $P_2$.  Suppose that there
    is a vertex $x$ different from $e_1, e_2, e'_1, e'_2$ in the
    attachment of $C$ over $L(H)$.  By Lemma~\ref{l:cycleedgek4}
    applied (in $H$) to edge $x$ and cycle $P_1\cup P_2$, $H$ contains
    a subdivision $H''$ of $K_4$ that contains $P_1\cup P_2 \cup
    \{x\}$.  By Lemma~\ref{l:compsk4}, $C$ is either of type branch,
    triangle, augmenting or square with respect to $L(H'')$.  In fact
    $C$ is not of type square as we are not in Case~1; moreover, $C$
    cannot be of type triangle or augmenting as it has at least five
    neighbors in $L(H'')$.  So it is of type branch.  But the branch
    of $H''$ containing $x$ is edgewise disjoint from $P_1\cup P_2$, a
    contradiction.  So $x$ does not exist, and we conclude that $C$ is
    of type augmenting with respect to $L(H)$.
\end{proof}

\begin{lemma}
  \label{l:substdecomp}
  Let $G$ be a graph that contains no triangular ISK4.  Let $H$ be a
  substantial graph such that $L(H)$ is an induced subgraph of $G$ and
  is inclusion-wise maximum with respect to that property.  Then
  either $G = L(H)$, or $G$ has a clique-cutset of size at most three,
  or $G$ has a proper $2$-cutset.
\end{lemma}
 
\begin{proof}
  Suppose that $G \neq L(H)$.  So there is a component $C$ of $G
  \setminus L(H)$.  Let us apply Lemma~\ref{l:compsubstan} to $C$ and
  $L(H)$.  Suppose that $C$ is of type augmenting.  So there is a path
  $P$ like in the definition of the type augmenting.  In $H$ the
  attachment of $C$ consists of four edge $ab$, $be$, $cd$, $df$,
  where $b, d$ have degree two in $H$.  Let us consider the graph $H'$
  obtained from $H$ by adding between $b$ and $d$ a path $R$ whose
  length is one plus the length of $P$.  Then $H'$ is substantial.
  Indeed, it is cyclically $3$-connected by Lemma~\ref{l:addedge}, and
  it is not a square theta or a square subdivision of $K_4$ since $H$
  is not a square theta.  Moreover, $L(H')$ is an induced subgraph of
  $G$, where $P$ corresponds to the path $R$.  This is a contradiction
  to the maximality of $L(H)$.  So $C$ is of type branch or triangle.
  If $C$ is of type branch, then the ends of the branch that contain
  the attachment of $C$ form a cutset of $G$ of size at most two.  So
  either this is a proper $2$-cutset or it contains a clique-cutset.
  If $C$ is of type triangle, then the triangle that contains the
  attachment of $C$ is a clique cutset of $G$.
\end{proof}


\section{Rich squares}

\begin{lemma}
  \label{l:richsquaredecomp}
  Let $G$ be an ISK4-free graph that does not contain the line graph
  of a substantial graph.  Let $K$ be a rich square that is an induced
  subgraph of $G$ and is maximal with respect to this property.  Then
  either $G=K$ or $G$ has a clique-cutset of size at most three or $G$
  has a proper $2$-cutset.
\end{lemma}

\begin{proof}
  Let $S$ be a central square of $K$, with vertices $u_1, u_2, u_3,
  u_4$ and edges $u_1u_2, u_2u_3, u_3u_4, u_4u_1$.  Recall that every
  component of $K\setminus S$ is a link of $S$.  A link with ends $p,
  p'$ is said to be \emph{short} if $p=p'$, and \emph{long} if $p\neq
  p'$.  Note that long links are flat branches of~$K$.  If two long
  links $B_1= p_1 \tp \cdots \tp p'_1$ and $B_2= p_2\tp \cdots \tp
  p'_2$ are such that $N_S(p_1)= N_S(p_2)$ and $N_S(p'_1)= N_S(p'_2)$,
  then we say that $B_1, B_2$ are \emph{parallel}, otherwise they are
  \emph{orthogonal}.
  
  Suppose that $G\neq K$.  Let $C$ be a component of $G \setminus K$.
  We may assume that the attachment of $C$ over $K$ is not empty, for
  otherwise any vertex of $K$ would be a cutset of~$G$.  This leads to
  the following three cases.
  
  \medskip

  \noindent{\it Case 1: $N_K(C)$ contains vertices of a long
  link of $S$.} Let $B_1 = p_1 \tp \cdots \tp p'_1$ be such a
  link.  We may assume up to symmetry that $N_S(p_1) = \{u_1,
  u_2\}$ and $N_S(p_1') = \{u_3, u_4\}$.  If $C$ has no neighbor in
  $K\setminus B_1$, then $\{p_1, p_1'\}$ is a proper $2$-cutset of~$G$
  and the lemma holds.  So $C$ has a neighbor in $K\sm B_1$.
  
  Suppose that $C$ has no neighbor in $K \setminus (S\cup B_1)$.
  Hence $C$ has a neighbor in $S$.  By Lemma~\ref{l:compprism} applied
  to the prism $S \cup B_1$ and $C$, we deduce that $C$ is of type
  augmenting, triangle or square.  If $C$ is of type triangle, then
  there is a triangle cutset in $G$, and the lemma holds.  If $C$ is
  of type augmenting, let $P$ be a shortest path of $C$ that sees
  $B_1$ and $S$.  Let $B$ be a component of $K\setminus (S\cup B_1)$.
  Then $G[B_1\cup B \cup P \cup \{u_1, u_3\}]$ is an ISK4, a
  contradiction.  If $C$ is of type square and not augmenting, then it
  must be that $B_1$ has length one and, up to symmetry, $C$ contains
  a path $P$ with one end adjacent to $u_1, p_1$ and the other end to
  $u_4, p'_1$.  Let $B$ be any component of $K\setminus (S\cup B_1)$.
  Then $G[B_1\cup B \cup P \cup \{u_1, u_3\}]$ is an ISK4, a
  contradiction.

  Therefore $N_K(C)$ contains vertices of a component $B_2$ of $K\sm
  (S\cup B_1)$.  Suppose that $B_2$ is either short or orthogonal to
  $B_1$.  Then $K' = G[S\cup B_1\cup B_2]$ is the line graph of a
  subdivision of $K_4$, and we can apply Lemma~\ref{l:compsk4} to $K'$
  and $C$.  Clearly, $C$ is not of type branch or triangle with
  respect to $K'$, and it is also not of type square because $B_1 \cup
  B_2$ contains no even square of $K'$.  So $C$ is of type
  augmenting, with a path $P$ as in the definition of type augmenting.
  This implies that $B_2$ is a flat branch of $K$, and so it is a
  long link of $S$.  Then $G[S \cup B_1 \cup B_2 \cup P]$ is the
  line graph of a substantial graph, a contradiction.

  So $B_2$ is a long link parallel to $B_1$.  Let $B_2 = p_2 \tp
  \cdots \tp p'_2$ with $N_S(p_2)=N_S(p_1)$ and $N_S(p'_2)=N_S(p'_1)$.
  Let $P= p_3 \tp \cdots \tp p'_3$ be a shortest path of $C$ such that
  $p_3$ has neighbors in $B_1$ and $p_3'$ has neighbors in $B_2$.  If
  no vertex of $P$ has a neighbor in $\{u_1, u_2\}$, then $B_1 \cup
  B_2 \cup P$ contains a path that closes the diamond $\{p_1, p_2,
  u_1, u_2\}$, a contradiction.  So some vertex of $P$ has a neighbor
  in $\{u_1, u_2\}$ and similarly some vertex of $P$ has a neighbor in
  $\{u_3, u_4\}$.  By Lemma~\ref{l:compprism} applied to the prism $K'
  = G[S \cup B_1]$ and $P$, we deduce that $P$ is of type augmenting
  with respect to $K'$.  Let $P'$ be a shortest subpath of $P$ that
  contains neighbors of $B_1$ and $S$.  One end of $P'$ must be $p_3$,
  and $N_{B_1}(p_3) = \{q_1, q'_1\}$, where $q_1 q'_1$ is an edge of
  $B_1$ and $p_1, q_1, q_1', p_1'$ appear in this order along $B_1$.
  We denote the other end of $P'$ by $p_3''$, and we can assume up to
  symmetry that $N_K(p_3'') = \{ u_2, u_3\}$.  If $p_3'\neq p_3''$,
  then $B_1 \cup B_2 \cup P' \cup \{u_1, u_3\}$ is a triangular ISK4,
  a contradiction.  So $p_3' = p_3''$.  By Lemma~\ref{l:compprism}
  applied to $K'' = G[S\cup B_2]$ and $p'_3$, we deduce that $p_3'$ is
  of type augmenting with respect to $K''$, so $N_{B_2}(p_3') = \{q_2,
  q'_2\}$, where $q_2, q'_2$ is an edge of $B_2$ and $p_2, q_2, q_2',
  p_2'$ appear in this order along $B_2$.  Then the paths $p'_3 \tp
  u_2$, $p'_3 \tp q_2 \tp B_2 \tp p_2$ and $p'_3 \tp P \tp p_3 \tp
  q'_1 \tp B_1 \tp p'_1 \tp u_4 \tp u_1$ form a triangular ISK4, a
  contradiction.

  \medskip

  \noindent{\it Case 2: $N_K(C)$ does not contain any vertex of a long
  link of $S$, and contains vertices of a short link.} So there exists
  a vertex $b_1$ adjacent to all of $S$ and to $C$.  Suppose that $C$
  is also adjacent to a component of $K\setminus (S\cup b_1)$, that
  is, to a vertex $b_2\neq b_1$ adjacent to all of $S$.  Then $K' =
  G[S \cup \{b_1, b_2\}]$ is the line graph of $K_4$, so we can apply
  Lemma~\ref{l:compsk4} to $K'$ and $C$.  We deduce that $C$ is of
  type square with respect to $K'$, with a linking path $P$.  Since $K
  \cup P$ cannot be a rich square (which would contradict the
  maximality of $K$), we may assume up to symmetry that $N_{K'}(P) =
  \{u_1, u_3, b_1, b_2\}$.  Since $K$ is a maximal rich square, and
  $S\cup P \cup \{b_1, b_2\}$ is a rich square, $K \setminus (S\cup
  \{b_1, b_2\})$ must have a component $B_3$ (a link of $S$).  Then
  $B_3\cup P \cup \{u_2, u_4, b_1, b_2\}$ is a (non-triangular) ISK4,
  a contradiction.  So no vertex of $C$ has a neighbor in $K\setminus
  (S\cup \{b_1\})$.  Let $B_2$ be any component of $K\setminus
  (S\cup\{b_1\})$.  Note that $K' = S\cup B_2 \cup \{b_1\}$ is the
  line graph of a subdivision of $K_4$.  By Lemma~\ref{l:compsk4}
  applied to $K'$ and $C$, we deduce that $C$ is of type triangle with
  respect to $K'$.  Since no vertex of $C$ has a neighbor in a
  component of $K\setminus S$ (except $b_1$), we see that $G$ has a
  triangle cutset.
  
    \medskip

  \noindent{\it Case 3: $N_K(C)$ is included in $S$.} Let $K'$ be a
  subgraph of $K$ that contains $S$ and is either the line graph of an
  ISK4 or a prism (take $S$ plus a long link if possible or two
  short links otherwise).  We can apply Lemma~\ref{l:compprism}
  or~\ref{l:compsk4} to $K'$ and $C$.  If $C$ is of type augmenting or
  square with respect to $K'$ with path $P$, then $K \cup P$ is a rich
  square, a contradiction to the maximality of $K$.  If $C$ is of type
  branch or triangle, then $G$ has a proper $2$-cutset or a clique
  cutset.
\end{proof}


\section{Prisms}

\begin{lemma}
  \label{l:prismdecomp}
  Let $G$ be an ISK4-free graph that does not contain the line graph
  of a substantial graph or a rich square as an induced subgraph.  Let
  $K$ be a prism that is an induced subgraph of $G$.  Then either
  $G=K$ or $G$ has a clique-cutset of size at most three or $G$ has a
  proper $2$-cutset.
\end{lemma}

\begin{proof}
  Suppose that $G\neq K$, and let $C$ be any component of $G\setminus
  K$.  Apply Lemma~\ref{l:compprism} to $K$ and $C$.  If $C$ is of
  type branch, then the ends of the branch of $K$ that contains the
  attachment of $C$ over $K$ is a cutset of size at most two, and
  either it is proper or it contains a clique cutset.  If $C$ is of
  type triangle, then $G$ has a triangle cutset.  If $C$ is of type
  augmenting, with augmenting path $P$, then $P \cup K$ is either the
  line graph of a non-square subdivision of $K_4$, or a rich square,
  in both cases a contradiction.  If $C$ is of type square, with a
  linking path $P$, then $K\cup P$ is a rich square, a contradiction.
\end{proof}

\begin{lemma}
  \label{l:mainlinegraphdecomp}
  Let $G$ be an ISK4-free graph that contains a prism.  Then either $G$
  is the line graph of a graph with maximum degree three,  or $G$ is a
  rich square,  or $G$ has a clique-cutset of size at most three or $G$
  has a proper $2$-cutset.
\end{lemma}

\begin{proof}
  Since $G$ contains a prism, it contains as an induced subgraph the
  line graph $L(H)$ of a cyclically $3$-connected graph.  By
  Lemma~\ref{l:descc3c}, $H$ is either a theta or a subdivision of a
  $3$-connected graph.  In the latter case, if $H$ is substantial,
  then the result holds by Lemma~\ref{l:substdecomp}.  Else, we may
  assume that $G$ does not contain the line graph of a substantial
  graph and $L(H)$ is a rich square made of a square with two links,
  and then the result holds by Lemma~\ref{l:richsquaredecomp}.  Hence,
  in the first case, we may assume that $G$ contains no rich square
  and no line graph of a substantial graph.  Then the result holds by
  Lemma~\ref{l:prismdecomp}.
\end{proof}


\section{Wheels and double star cutset}

A \emph{paw} is a graph with four vertices $a,  b,  c,  d$ and four edges
$ab,  ac,  ad,  bc$.

\begin{lemma}\label{thm:paw}
  Let $G$ be a graph that does not contain a triangular ISK4 or a
  prism.  If $G$ contains a paw,  then $G$ has a star-cutset.
\end{lemma}
\begin{proof}
  Suppose that $G$ does not have a star-cutset.  Let $X$ be a paw in
  $G$, with vertices $a, b, c, d$ and edges $ab, ac, ad, bc$.  Since
  $G$ does not admit a star-cutset, the set $\{a\}\cup N(a)\setminus
  \{b, d\}$ is not a cutset of $G$, so there exists a chordless path
  $P_1$ with endvertices $b, d$ such that the interior vertices of
  $P_1$ are distinct from $a$ and not adjacent to $a$.  Likewise, the
  set $\{a\}\cup N(a)\setminus \{c, d\}$ is not a cutset of $G$, so
  there exists a chordless path $P_2$ with endvertices $c, d$ such
  that the interior vertices of $P_2$ are distinct from $a$ and not
  adjacent to $a$.  The definition of $P_1, P_2$ implies that there
  exists a path $Q$ with endvertices $b, c$ such that $V(Q)\subseteq
  V(P_1)\cup V(P_2)$, $Q$ is not equal to the edge $bc$, and $bc$ is
  the only chord of $Q$.  So $V(Q)$ induces a cycle.  If $d$ is in
  $Q$, then $V(Q)\cup \{a\}$ induces a triangular subdivision of
  $K_4$, a contradiction.  If $d$ is not in $Q$, then the definition
  of $P_1, P_2$ implies that there exists a path $R$ whose endvertices
  are $d$ and a vertex $q$ of $Q$ and $V(R)\subseteq V(P_1)\cup
  V(P_2)$.  We choose a minimal such path $R$.  Let $d'$ be the
  neighbor of $q$ in $R$.  the minimality of $R$ implies that $R$ is
  chordless, $(V(R)\setminus \{q\}) \cap V(Q)=\emptyset$, and $d'$ is
  the only vertex of $R$ with a neighbor in $Q$.  If $d'$ has only one
  neighbor in $Q$, then $V(Q)\cup V(R)\cup \{a\}$ induces a triangular
  subdivision of $K_4$ (whose corners are $a, b, c, q$), a
  contradiction.  If $d'$ has exactly two neighbors in $Q$ and these
  are adjacent, then $V(Q)\cup V(R)\cup \{a\}$ induces a prism, a
  contradiction.  If $d'$ has at least two non-adjacent neighbors in
  $Q$, then $V(Q)\cup V(R)\cup \{a\}$ contains an induced triangular
  subdivision of $K_4$ (whose corners are $a, b, c, d'$), a
  contradiction.
\end{proof}

\begin{lemma}\label{thm:hu4}
  Let $G$ be an ISK4-free graph that does not contain a prism or an
  octahedron.  If $G$ contains a wheel $(H,  u)$ with $|V(H)|=4$,  then
  $G$ has a star-cutset.
\end{lemma}
\begin{proof}
  Suppose that $G$ does not have a star-cutset.  Let the vertices of
  $H$ be $u_1, \ldots, u_4$ in this order.  If $u$ is adjacent to only
  three of them, then $V(H)\cup \{u\}$ induces a subdivision of $K_4$.
  So we may assume that $u$ is adjacent to all vertices of $H$.  Since
  $G$ does not admit a star-cutset, the set $\{u\}\cup N(u)\setminus
  \{u_1, u_3\}$ is not a cutset of $G$, so there exists a chordless
  path $P$ with endvertices $u_1, u_3$ such that the interior vertices
  of $P$ are distinct from $u$ and not adjacent to $u$.  Let $P=u_1
  \tp v \tp \cdots \tp u_3$.  Vertex $v$ must be adjacent to $u_2$,
  for otherwise $\{u, u_1, u_2, v\}$ induces a paw, which contradicts
  Lemma~\ref{thm:paw}.  Likewise, $v$ is adjacent to $u_4$.  If $v$ is
  not adjacent to $u_3$, then $\{u_1, u_2, u_3, u_4, v\}$ induces a
  subdivision of $K_4$, a contradiction.  If $v$ is adjacent to $u_3$,
  then $\{u, u_1, u_2, u_3, u_4, v\}$ induces an octahedron, a
  contradiction.
\end{proof}

\begin{lemma}\label{thm:wheel}
  Let $G$ be an ISK4-free graph that does not contain a prism or an
  octahedron.  If $G$ contains a wheel,  then $G$ has a star-cutset or
  a double star cutset.
\end{lemma}
\begin{proof}
Suppose that the lemma does not hold.  Let $(H, u)$ be a wheel in $G$
such that $|V(H)|$ is minimum.
Let $u_1, \ldots, u_h$ be the neighbors of $u$ in $H$ in this order.
If $h=3$, then $V(H)\cup\{u\}$ induces a subdivision of $K_4$, so we
may assume that $h\ge 4$.  By Lemma~\ref{thm:hu4}, we may assume that
$|V(H)|\ge 5$.  Let us call \emph{fan} any pair $(P, x)$ where $P$ is
a chordless path, $x$ is a vertex not in $P$, and $x$ has exactly four
neighbours in $P$, including the two endvertices of $P$.  Since
$|V(H)|\ge 5$, we may assume up to symmetry that $u_1$ and $u_4$ are
not adjacent.  Letting $Q$ be the subpath of $H$ whose endvertices are
$u_1, u_4$ and which contains $u_2, u_3$, we see that $(Q, u)$ is a
fan.  Since $G$ contains a fan, we may choose a fan $(P, x)$ with a
shortest $P$.  Let $x_1, x_2, x_3, x_4$ be the four neighbours of $x$
in $P$ in this order, where $x_1, x_4$ are the endvertices of $P$.  If
$x_1$ is adjacent to $x_2$, then $\{x, x_1, x_2, x_4\}$ induces a paw,
which contradicts Lemma~\ref{thm:paw}.  So $x_1$ is not adjacent to
$x_2$, and similarly $x_3$ is not adjacent to $x_4$.  Also $x_2$ is
not adjacent to $x_3$, for otherwise $\{x, x_1, x_2, x_3\}$ induces a
paw.  For $i=1, 2, 3$, let $P_i$ be the subpath of $P$ whose
endvertices are $x_i$ and $x_{i+1}$.  Let $x'_2, x''_2$ be the two
neighbours of $x_2$ in $P$, such that $x_1, x'_2, x_2, x''_2, x_3,
x_4$ lie in this order in $P$.


Since $G$ does not admit a double star cutset, the set $\{x, x_2\}\cup
N(x)\cup N(x_2)\setminus \{x'_2, x''_2\}$ is not a cutset, and so
there exists a path $Q=v_1 \tp \cdots \tp v_k$ such that $v_1$ has a
neighbour in the interior of $P_1$, $v_k$ has a neighbour in the
interior of $P_2$, and the vertices of $Q$ are not adjacent to either
$x$ or $x_2$.  We may choose a shortest such path $Q$, so $Q$ is
chordless and its interior vertices have no neighbour in $V(P_1)\cup
V(P_2)$.  If $v_1$ has at least four neighbours in $P_1$, then there
is a subpath $P'_1$ of $P_1$ such that $(P'_1, v_1)$ is a fan, which
contradicts the minimality of $(P, x)$.  If $v_1$ has exactly three
neighbours in $P_1$, then $V(P_1)\cup \{x, v_1\}$ induces a
subdivision of $K_4$.  So $v_1$ has at most two neighbours in $P_1$.
Let $\{y_1, z_1\}$ be the set of neighbours of $v_1$ in $P_1$, such
that $x_1, y_1, z_1, x_2$ lie in this order in $P_1$ (possibly
$y_1=z_1$).  Likewise, $v_k$ has at most two neighbours in $P_2$.  Let
$\{y_2, z_2\}$ be the set of neighbours of $v_k$ in $P_2$, such that
$x_2, y_2, z_2, x_3$ lie in this order in $P_2$ (possibly $y_2=z_2$).

Suppose that $y_1\neq z_1$.  Note that $z_1$ and $z_2$ are not
adjacent, for that would be possible only if $z_1=x_2$ (and
$z_2=x''_2$), which would contradict the definition of $Q$.  Then
$V(P_1) \cup V(z_2-P_2-x_3)\cup V(Q) \cup \{x\}$ induces a subdivision
of $K_4$.  So $y_1=z_1$.  Likewise, $y_2=z_2$.  But, then $V(P_1)\cup
V(P_2)\cup V(Q)\cup \{x\}$ induces a subdivision of $K_4$.
\end{proof}


\section{Decomposition theorems}
\label{sec:th:main}

\subsubsection*{Proof of Theorem~\ref{th:nowheel}}

Let $G$ be a graph that contains no ISK4 and no wheel.  By
Lemma~\ref{l:begin}, we may assume that $G$ contains a $K_{3, 3}$ or a
prism.  Note that $G$ cannot be a thick complete tripartite graph,
because such a graph contains a wheel $K_{1, 2, 2}$.  So if $G$
contains $K_{3, 3}$, then we are done by Lemma~\ref{l:decK33}.  If $G$
contains a prism, then we are done by
Lemma~\ref{l:mainlinegraphdecomp}.

\subsubsection*{Proof of Theorem~\ref{th:main}}

  By Theorem~\ref{th:nowheel}, we can assume that $G$ is either a
  complete bipartite graph, a rich square or contains a wheel.  Note
  that complete bipartite graphs and rich square either are
  series-parallel or admit a star cutset or a double star cutset.  So
  we may assume that $G$ contains a wheel.  If $G$ contains a prism
  then we are done by Lemma~\ref{l:mainlinegraphdecomp}.  So, we
  assume that $G$ contains no prism and in particular no line graph of
  a substantial graph.  If $G$ contains an octahedron, then we are
  done by Lemma~\ref{l:richsquaredecomp}, since an octahedron is a
  rich square.  So we may assume that $G$ contains no prism and no
  octahedron.  Hence, we are done by Lemma~\ref{thm:wheel}.


\section{Chordless graphs}
\label{sec:thm:nochord}

Most of the proof of Theorem~\ref{thm:nochord} is implicitly given
in~\cite{nicolas.kristina:one} (proof of Theorem~2.2 and Claims~12
and~13 in the proof of Theorem~2.4).  But the result is not stated
explicitly in~\cite{nicolas.kristina:one} and many details differ.
For the sake of completeness and clarity we repeat the whole argument
here.

\subsubsection*{Proof of Theorem~\ref{thm:nochord}}

  Let us assume that $G$ has no $1$-cutset and no proper $2$-cutset.
  Note that $G$ contains no $K_4$,  since a $K_4$ is a cycle with two
  chords.  Moreover:
 
\setcounter{claim}{0}

  \begin{claim}
    \label{c:triangle}
    We may assume that $G$ is triangle-free. 
  \end{claim}
  
  \begin{proofclaim}
    For suppose that $G$ contains a triangle $T$.  Then $T$ is a
    maximal clique of $G$ since $G$ contains no $K_4$.  We may assume
    that $G \neq T$ because a triangle is sparse,  and that $G$ is
    connected,  for otherwise every vertex is a $1$-cutset.  So some
    vertex $a$ of $T$ has a neighbor $x$ in $G\setminus T$.  Since $a$
    is not a $1$-cutset of $G$,  there exists a shortest path $P$
    between $x$ and a member $b$ of $T\sm a$.  But,  then $P\cup T$ is a
    cycle with at least one chord (namely $ab$),  a contradiction.
  \end{proofclaim}

  \begin{claim}
    \label{c:cliqueCutset}
    We may assume that $G$ has no clique cutset.
  \end{claim}
  
  \begin{proofclaim}
    Suppose that $K$ is a clique cutset in $G$.  Since $G$ has no
    cutset of size one and there is no clique of size at least three
    by~(\ref{c:triangle}),  $K$ has exactly two elements $a$ and $b$.
    Let $X$ and $Y$ be two components of $G\sm \{a,  b\}$.  Since none
    of $a$ and $b$ is a 1-cutset of $G$,  $X\cup\{a,  b\}$ contains a
    path $P_X$ with endvertices $a$ and $b$; and a similar path $P_Y$
    exists in $Y\cup\{a, b\}$.  But,  then $P_X\cup P_Y$ forms a cycle
    with at least one chord (namely $ab$),  a contradiction.
  \end{proofclaim}

   We can now prove that $G$ is sparse.  Suppose on the contrary that
   $G$ has two adjacent vertices $a,b$ both of degree at least three.
   Let $c, e$ be two neighbors of $a$ different from $b$, and let $d,
   f$ be two neighbors of $b$ different from $a$.  Note that $\{c,
   e\}$ and $\{d, f\}$ are disjoint by~(\ref{c:triangle}).
   By~(\ref{c:cliqueCutset}), $\{a, b\}$ is not a cutset, so there is
   in $G\setminus\{a, b\}$ a path between $\{c, e\}$ and $\{d, f\}$
   and consequently a path $P$ that contains exactly one of $c, e$ and
   one of $d, f$.  Let the endvertices of $P$ be $e$ and $f$ say.
   Thus $P \cup\{a, b\}$ forms a cycle $C$.  Since $G\setminus\{a,
   b\}$ is connected, there exists a path $Q = c \tp \cdots \tp u$,
   where $u \in P\cup \{b, d\}$ and no interior vertex of $Q$ is in $C
   \cup \{d\}$.  If $u$ is in $\{b, d\}$, then $Q\cup C$ forms a cycle
   with at least one chord, namely $ab$.  So $u\in P$.  Also since $G
   \setminus \{a, b\}$ is connected, there exists a path $R = d \tp
   \cdots \tp v$ where $v \in P \cup Q$ and no interior vertex of $R$
   is in $C \cup Q$.

  If $v$ is in $Q \sm u$, then $bdRvQcaePfb$ is a cycle with at least
  one chord, namely $ab$, a contradiction.  So $v$ is in $P$.  If $e,
  v, u, f$ lie in this order on $P$ and $v\neq u$, then
  $bdRvPeacQuPfb$ is a cycle with at least one chord, namely $ab$, a
  contradiction.  So $e, u, v, f$ lie in this order on $P$ (possibly
  $u=v$).  This restores the symmetry between $c$ and $e$ and between
  $d$ and $f$.  We suppose from here on that the paths $P, Q, R$ are
  chosen subject to the minimality of the length of $uPv$.

  Let $P_e=ePu\sm u$, $Q_c=cQu\sm u$, and $P_b = bPu\sm u$.  We show
  that $\{a, u\}$ is a $2$-cutset of $G$.  Suppose not; so there is a
  path $D=x \tp \cdots \tp y$ in $G \setminus \{a, u\}$ such that $x$
  lies in $P_e \cup Q_c$, $y$ lies in $P_b \cup R$, and no interior
  vertex of $D$ lies in $P\cup \{a\}\cup Q\cup R$.  we may assume up
  to symmetry that $x$ is in $Q_c$.  If $y$ is in the subpath $u \tp P
  \tp v$, then, considering path $Q'=c \tp Q \tp x \tp D \tp y$, we
  see that the three paths $P, Q', R$ contradict the choice of $P, Q,
  R$ because $y$ and $v$ are closer to each other than $u$ and $v$
  along $P$.  So $y$ is not in $uPv$, and so, up to symmetry, $y$ is
  in $R \setminus \{ v\}$.  But, then $x Q a e P f b R y D x$ is a
  cycle with at least one chord (namely $ab$), a contradiction.  This
  proves that we can partition $G\sm\{a, u\}$ into a set $X$ that
  contains $P_e\cup Q_c$ and a set $Y$ that contains $P_b\cup R$ such
  that there is no edge between $X$ and $Y$, so $\{a, u\}$ is a
  $2$-cutset.  So, by~(\ref{c:cliqueCutset}), $a$ and $u$ are not
  adjacent.  This implies that $\{a, u\}$ is proper.



\section{Forbidding wheels}
\label{s:WF3c}

Recall that a \emph{branch} in a graph $G$ is a path of $G$ of length
at least one whose ends are branch vertices and whose internal
vertices are not (so they all have degree~2).  A \emph{subbranch} is a
subpath of a branch.  \emph{Reducing} a subbranch of length at least
two means replacing it by an edge.

\begin{lemma}
  \label{l:WF1}
Let $G$ be a graph that contains no ISK4,  no wheel and no $K_{3, 3}$.
Let $B$ be a subbranch of length at least two in $G$,  and let $G'$ be
the graph obtained from $G$ by reducing $B$.  Then $G'$ contains no
ISK4,  no wheel and no $K_{3, 3}$.
\end{lemma}

\begin{proof}
Let $e$ be the edge of $G'$ that results from the reduction of $B$.

Suppose that $G'$ contains an ISK4 $H$.  Then $H$ must contain $e$, 
for otherwise $H$ is an ISK4 in $G$.  Then replacing $e$ by $B$ in $H$
yields an ISK4 in $G$,  a contradiction.

Now suppose that $G'$ contains a wheel $W=(H,  x)$.  Let $x_1,  \ldots, 
x_h$ be the neighbors of $x$ in $H$,  with $h\ge 4$.  Then $W$ must
contain $e$,  for otherwise $W$ is a wheel in $G$.  Suppose that $e$ is
an edge in $H$.  Then replacing $e$ by $B$ in $H$ yields a wheel in
$G$ (with hub $x$ and the same number of spokes),  a contradiction.
Now suppose that $e=xx_h$.  So,  in $G$,  vertices $x$ and $x_h$ are the
endvertices of $B$ and they are not adjacent.  If $h\ge 5$,  then $(H, 
x)$ induces a wheel in $G$ (with the same hub and with $h-1$ spokes).
If $h=4$,  then $V(H)\cup\{x\}$ induces an ISK4 in $G$,  a
contradiction.

Finally, suppose that $G'$ contains a $K_{3, 3}$ $H$.  Then $H$ must
contain $e$, for otherwise $H$ is a $K_{3, 3}$ in $G$.  Let $e=xy$.
Then $x$ and $y$ are the endvertices of $B$ in $G$ and they are not
adjacent, so $V(H)$ induces an ISK4 in $G$, a contradiction.
\end{proof}

Note that the converse of Lemma~\ref{l:WF1} is not true.  Let $G$ be
the graph with vertices $x_1, \ldots, x_7$ such that $x_1, \ldots,
x_5$ induce a hole in this order, $x_6$ is adjacent to $x_1, x_3,
x_5$, and $x_7$ is adjacent to $x_2, x_4$.  Then $x_2$-$x_7$-$x_4$ is
a branch whose reduction yields the prism on six vertices, a graph
that contains no ISK4, no wheel and no $K_{3, 3}$.  But $G$ contains
an ISK4.

The following result is well-known.  See \cite{seymour:sp} for a
simple greedy coloring algorithm.

\begin{lemma}[Dirac,  \cite{dirac:SP}]
  \label{l:WF3c-sp}
Let $G$ be a series-parallel graph.  Then $G$ is $3$-colorable.
\end{lemma}

\begin{lemma}
  \label{l:WF3c-rs}
  Let $G$ be a rich square that contains no wheel.  Then $G$ is
  $3$-colorable.
\end{lemma}
\begin{proof}
 By the definition of a rich square, there is a square $S = \{u_1,
 u_2, u_3, u_4\}$ in $G$ such that every component of $G\setminus S$
 is a link of $S$.  We make a $3$-coloring of the vertices of $G$ as
 follows.  Assign color~$1$ to $u_1$, color~$2$ to $u_2$ and $u_4$,
 and color~$3$ to $u_3$.  Let $P$ be any component of $G\setminus S$.
 So $P$ is a path $p_1 \tp \cdots \tp p_t$.  Note that $t\ge 2$, for
 otherwise $S\cup \{p_1\}$ would induce a wheel (with four spokes).
 We may assume that $N_S(p_1) = \{u_1, u_2\}$ or $\{u_1, u_4\}$ and
 $N_S(p_t) = \{u_3, u_4 \}$ or $\{u_2, u_3\}$.  In either case, assign
 color~$3$ to $p_1$, color $1$ to $p_t$, and, if $t\ge 3$, assign
 colors $2$ and $3$ alternately to $p_2, \ldots, p_{t-1}$.  Repeating
 this for every link produces a $3$-coloring of the vertices of $G$.
\end{proof}

Note that Lemma~\ref{l:WF3c-rs} is tight, in the sense that a rich
square may fail to be $3$-colorable, as shown by the graph on Figure
\ref{fig:C6barPlusV}.  The following result also is tight since the
graph represented on Figure~\ref{fig:C6barPlusV} is a line graph.  The
line graph of the Petersen graph is another example of a line graph of
a cubic graph whose chromatic number is~4.

\begin{figure}[h]
      \centering
\includegraphics[scale=0.5]{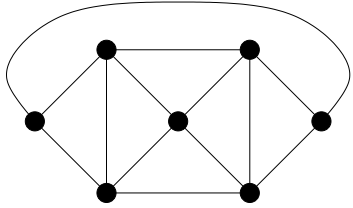}
\caption{Example of a rich square with chromatic number 4}
\label{fig:C6barPlusV}
    \end{figure}

\begin{lemma}
  \label{l:WF3c-lg}
  Let $G$ be a graph that contains no ISK4,  no wheel and such that $G$
  is a line graph.  Then $G$ is $3$-colorable.
\end{lemma}
\begin{proof}
  Let $G$ be the line graph of $H$.  So we need only to prove that $H$
  is $3$-edge-colorable.  Since $G$ contains no ISK4, in particular it
  contains no $K_4$, so $H$ has maximum degree at most three.  If $C$
  is a cycle of length at least four in $H$ and $e$ is a chord of $C$,
  then the edges of $C$ plus edge $e$ are vertices of $G$ that induce
  a wheel in $G$ (with hub $e$ and four spokes), a contradiction.  So
  every cycle of $H$ is chordless.  By Theorem~\ref{thm:nochord}, one
  of the following holds: \\
  (a) The vertices of $H$ of degree at least $3$ are pairwise
  non-adjacent; \\
  (b) $H$ has a  cutvertex; \\
  (c) $H$ has a proper $2$-cutset.\\
 
  We prove that our graph $H$ is $3$-edge-colorable in each case.

  (a) Let $f=xy$ be any edge of $H$.  Since $H$ satisfies (a), we may
  assume that $x$ has degree at most two and $y$ has degree at most
  three in $H$.  Thus, in $G$, vertex $f$ has degree at most three.
  It follows from the theorem of Brooks \cite{brooks} that $G$ is
  $3$-colorable (and so $H$ is $3$-edge-colorable).
 
  (b) Let $x$ be a cutvertex of $H$.  Let $A_1, \ldots, A_k$ be the
  components of $H\setminus x$, and let $H_i$ be the subgraph of $H$
  induced by $V(A_i)\cup\{x\}$ for each $i=1, \ldots, k$.  Since $H$
  is connected, $x$ has a neighbor in each $A_i$, and we have $k\le 3$
  since $H$ has maximum degree at most $3$.  By the induction
  hypothesis, each $H_i$ admits a $3$-edge-coloring.  Up to renaming
  some color classes, we can combine these colorings so that the
  colors used at $x$ are different; thus we obtain a $3$-edge-coloring
  for $H$.

  (c) Let $A_1, \ldots, A_k$ be the components of $H\setminus \{a,
  b\}$.  We may assume that we are not in case (b), so $H$ is
  $2$-connected and each of $a$ and $b$ has a neighbor in $A_i$ for
  each $i=1, \ldots, k$.  Since $H$ has maximum degree at most $3$, we
  may assume up to symmetry that $a$ has only one neighbor $a_1$ in
  $A_1$.  Suppose that $b$ has two neighbors in $A_1$.  Then $k=2$ and
  $b$ has only one neighbor $b_2$ in $A_2$, and , then $\{a, b_2\}$ is
  also a proper $2$-cutset of $H$.  Thus in any case we may assume
  that both $a, b$ have only one neighbor in $A_1$.  Let $b_1$ be the
  neighbor of $b$ in $A_1$.  Let $H_1$ be the graph obtained from
  $A_1$ by adding a vertex $x_1$ adjacent to $a_1$ and $b_1$.  Let
  $H_2$ be the graph obtained from $H\setminus A_1$ by adding a vertex
  $x_2$ adjacent to $a$ and $b$.  Suppose that $H_1$ contains a cycle
  $C$ that has a chord.  Then $C$ must contains $x_1$.  Since $H$ is
  $2$-connected there exists a chordless path $P$ with endvertices $a$
  and $b$ in $H\setminus A_1$.  Then $(C\setminus x)\cup P$ is a cycle
  with a chord in $H$, a contradiction.  So every cycle in $H_1$ is
  chordless.  By a similar argument, every cycle in $H_2$ is
  chordless.  Note that $H_1$ and $H_2$ have strictly fewer vertices
  than $H$ because the cutset $\{a, b\}$ is proper.  By the induction
  hypothesis, $H_1$ and $H_2$ have a $3$-edge-coloring.  In the
  coloring of $H_1$, edges $x_1a_1$ and $x_1b_1$ have different
  colors, and in the coloring of $H_2$ edges $x_2a$ and $x_2b$ have
  different colors too, so we can combine these colorings to make a
  $3$-edge-coloring for $H$.
\end{proof}

\label{sec:l:WF3c}
\subsubsection*{Proof of Theorem~\ref{l:WF3c}}

We prove the theorem by induction on the number of vertices of $G$.
Suppose that $G$ has a clique cutset $K$.  So $V(G)\setminus K$ can be
partitioned into two sets $X, Y$ such that there is no edge between
them.  Since $G$ contains no ISK4,  we have $|K|\le 3$.  By the
induction hypothesis,  the two subgraphs of $G$ induced by $X\cup K$
and $Y\cup K$ are $3$-colorable.  We can combine $3$-colorings of
these subgraphs so that they coincide on $K$,  and consequently we
obtain a $3$-coloring of $G$.  Now we may assume that $G$ has no
clique cutset.  If $G$ contains a $K_{3, 3}$,  then,  by
Lemma~\ref{l:decK33},  $G$ is a complete bipartite (recall that a thick
complete tripartite graph contains a wheel),  so it is $3$-colorable.
Now we may assume that $G$ contains no $K_{3, 3}$.

Suppose that $G$ has a $2$-cutset $\{a,  b\}$.  So $V(G)\setminus K$
can be partitioned into two sets $X, Y$ such that there is no edge
between them.  Since $G$ has no clique cutset,  it is $2$-connected,  so there
exists a chordless path $P_Y$ with endvertices $a$ and $b$ and with
interior vertices in $Y$.  Let $G'_X$ be the subgraph of $G$ induced
by $X\cup V(P_Y)$.  Note that $P_Y$ is a subbranch in $G'_X$.  Let
$G''_X$ be obtained from $G'_X$ be reducing $P_Y$ (thus $a$ and $b$
are adjacent in $G''_X$).  Define a graph $G''_Y$ similarly.  Since
$G'_X$ is an induced subgraph of $G$,  it contains no ISK4,  no wheel
and no $K_{3, 3}$.  So,  by Lemma~\ref{l:WF1},  $G''_X$ contains no ISK4, 
no wheel,  and no $K_{3, 3}$.  The same holds for $G''_Y$.  By the
induction hypothesis,  $G''_X$ and $G''_Y$ admit a $3$-coloring.  We
can combine these two $3$-colorings so that they coincide on $\{a, 
b\}$,  and consequently we obtain a $3$-coloring of $G$.

Now we may assume that $G$ contains no $2$-cutset.  By
Theorem~\ref{th:nowheel}, $G$ is either a series-parallel graph, a
rich square, a line graph, or a complete bipartite graph.
 Then the desired result follows from Lemmas~\ref{l:WF3c-sp},
\ref{l:WF3c-rs}, \ref{l:WF3c-lg}, and the fact that bipartite graphs
are 3-colorable.


\section{Algorithms for \{ISK4, wheel\}-free and chordless graphs }
\label{sec:algo}

In this section,  we give two algorithms for the class of \{ISK4, 
wheel\}-free graphs.  The first one is a recognition algorithm for
that class and the second is a coloring algorithm.  Both are based on
the results proved in the preceding sections.

    \subsection{Recognizing \{ISK4, wheel\}-free graphs}

    The recognition algorithm is based on Theorem~\ref{th:nowheel}: if
    a graph $G$ is \{ISK4, wheel\}-free, then either $G$ has a
    clique-cutset or a proper $2$-cutset, or $G$ is of one of the
    following four types: $G$ is series-parallel, $G$ is the
    line graph of a chordless graph with maximum degree at most three,
    $G$ is a complete bipartite graph, or $G$ is a rich square.  We
    analyze each of these cases separately.  Let us assume that $G$
    has $n$ vertices and $m$ edges.

    Suppose that $G$ has a clique cutset $K$.  So $V(G)\setminus K$
    can be partitioned into two sets $X, Y$ such that there is no edge
    between them.  Let $G_X$ and $G_Y$ be the subgraphs of $G$ induced
    by $X\cup K$ and $Y\cup K$.  We consider that $G$ is decomposed
    into $G_X$ and $G_Y$.  These subgraphs can in turn be decomposed
    along clique cutsets.  This is applied as long as possible, which
    yields a {\it clique cutset decomposition tree} $T_{cc}(G)$ of
    $G$.  Building such a tree can be done in time $O(n+m)$,
    see~\cite{Tar, Whi}.  If any clique cutset found during this step
    has size at least four, we stop with the obvious answer ``$G$ is
    not ISK4-free''.  Therefore let us assume that all the clique
    cutsets found by the algorithm have size at most three.  Note that
    a graph that is either a subdivision of $K_4$ or a wheel has no
    clique cutset.  It follows that $G$ is \{ISK4, wheel\}-free if and
    only if all leaves of $T_{cc}$ are \{ISK4, wheel\}-free.  So our
    algorithm proceeds with examining the leaves of the tree.

    Now suppose that $G$ has no clique cutset and has a proper
    $2$-cutset $\{a, b\}$.  So $V(G)\setminus \{a, b\}$ can be
    partitioned into two sets $X, Y$ such that there is no edge
    between them and each of $G[X \cup \{ a, b \}]$ and $G[Y \cup \{
    a, b \}]$ is not an $(a, b)$-path.  Let $G_X$ be the subgraph of
    $G$ induced by $X\cup \{a, b\}$ plus an artificial vertex adjacent
    to $a$ and $b$, and define $G_Y$ similarly.  Thus $G$ is
    decomposed into graphs $G_X$ and $G_Y$.  Note that $G_X$ and $G_Y$
    have fewer vertices than $G$ (because $\{a, b\}$ is proper), and
    that they have no clique cutset (because such a set would also be
    a clique cutset of $G$).  These subgraphs can in turn be
    decomposed along proper $2$-cutsets, and this is applied as long
    as possible, which yields a proper $2$-cutset decomposition tree
    $T_{2c}$ of $G$.  Note that a graph that is either a subdivision
    of $K_4$ or a wheel has no proper $2$-cutset.  It follows that $G$
    is \{ISK4, wheel\}-free if and only if all leaves of $T_{2c}$ are
    \{ISK4, wheel\}-free.  So our algorithm proceeds with examining
    the leaves of the tree.

    Let $T$ be the decomposition tree that is obtained by combining
    $T_{cc}(G)$ and the $T_{2c}$'s of all leaves of $T_{cc}$.  We show
    that $T$ has $O(n)$ nodes.  To do this, we define for every graph
    $H$ the function $f(H) = |V(H)|-4$.  Suppose that $G$ is
    decomposed by a cutset $K$ into subgraphs $G_X, G_Y$ as above,
    where $K$ is either a clique cutset of size at most three or a
    proper $2$-cutset.  If $K$ is a clique cutset, then we have
    $f(G_X) = |X| + |K| - 4$, $f(G_Y) = |Y| + |K| - 4$, and $f(G) =
    |X| +|Y| + |K| - 4$.  It follows (because $|K|\le 3$) that $f(G_X)
    + f(G_Y) \le f(G)$.  If $K$ is a proper $2$-cutset, then we have
    $f(G_X) = |X| + 3- 4$, $f(G_Y) = |Y| + 3- 4$, and $f(G) = |X| +
    |Y| + 2- 4$.  It follows again that $f(G_X) + f(G_Y) \le f(G)$.
    Let $T^*$ be the subtree of $T$ induced by the nodes that are
    graphs with at least five vertices.  Applying the above inequality
    recursively, and letting $G_1, \ldots, G_{\ell}$ be the leaves of
    $T^*$, we obtain that $f(G_1)+\cdots +f(G_\ell)\le f(G)$.  Since
    all $G_i$'s satisfy $f(G_i)>0$, we obtain $\ell \le n$.
    Consequently, $T^*$ has at most $2n-1$ nodes.  In addition, each
    node of $T$ with at least five vertices may have one or two
    children with at most four vertices.  Moreover, the size of the
    decomposition tree of graphs with at most four vertices is bounded
    by a constant.  So $T$ has $O(n)$ leaves.  Recall that the leaves
    have fewer vertices than $G$.

    Now we show that $T$ can be constructed in time $O(n^2m)$.
    Because proper 2-cutset can be found in time $O(nm)$ as follows:
    for any vertices $v$, find the cut vertices and the blocks of
    $G\sm v$ by using DFS.  For any such block, check whether the
    corresponding cutvertex $u$ is such that $\{u, v\}$ is a proper
    2-cutset.  Thus, building the tree can be done by running $O(n)$
    times this subroutine (or the routine that finds a clique cutset)
    and therefore takes time $O(n^2m)$.

    Now suppose that $G$ has no clique cutset and no proper
    $2$-cutset.  Theorem~\ref{th:nowheel} implies that if $G$ contains
    no induced subdivision of $K_4$ and no wheel, then $G$ must be
    either (i) series-parallel, or (ii) a complete bipartite graph, or
    (iii) a long rich square or (iv) the line graph of a chordless
    graph $H$ with maximum degree at most three.  The converse is also
    true, namely, if $G$ satisfies one of (i)--(iv), then it contains
    no ISK4 and no wheel (this is easy to check and we omit the
    details).  So our algorithm needs only test if $G$ is of one of
    the four types.

    Testing (i) can be done in time $O(n+m)$,  see~\cite{VTL}.

    Testing (ii) can be done by checking with breadth-first search
    whether $G$ is bipartite, and, then checking whether any two
    vertices on different sides of the bipartition are adjacent.  This
    takes time $O(m+n)$.

    To test (iii), note that if $G$ is a rich square and contains no
    wheel, then $G$ has exactly four vertices of degree at least four
    (the four vertices of the central square) and all other vertices
    have degree three or two.  So we need only identify the four
    vertices of largest degree, check whether they induce a square
    $S$, and, then check whether each component of $G\setminus S$ is a
    path and attaches to $S$ as in the definition of a rich square.
    This can be done in time $O(n+m)$.
  
    In order to test (iv), we apply one of the algorithms
    in~\cite{Leh, Rouss}, which run in time $O(n+m)$.  If $G$ is a
    line graph, then any such algorithm returns a graph $H$ such that
    $G$ is the line graph of $H$; moreover, it is known that $H$ is
    unique up to isomorphism, except when $G$ is a clique on three
    vertices (where $H$ is either $K_3$ or $K_{1, 3}$).  Then we need
    only check if $H$ has maximum degree a most three, which is easy,
    and contains no cycle with a chord, which can be done in time
    $O(n^2m)$ by a method described in the next section.  

    Let us now evaluate the total complexity of the algorithm.
    Building the tree takes time $O(n^2m)$.  Since for each leaf
    $H$ on $n'$ vertices and $m'$ edges, the test performed on $H$
    takes time $O(n'^2m')$, and since the sum of the sizes of the
    leaves of the tree is $O(n+m)$, processing all the leaves of the tree
    takes time $O(n^2m)$.  Hence, the recognition algorithm runs
    in time $O(n^2m)$.
    
    We would have liked to make our algorithm rely on classical
    decomposition along 2-cutsets, but the classical algorithms, such
    as Hopcroft and Tarjan's decomposition into triconnected
    components \cite{hopcroft.tarjan:3con}.  But this algorithm does
    not use our ``proper'' 2-cutset, so we do not know how we could
    use it.

\subsection{Recognizing and coloring chordless graphs}

On the basis of Theorem~\ref{thm:nochord}, we can give a
polynomial-time recognition algorithm for chordless graphs.  We
describe this algorithm informally.  Let the input of the algorithm be
a graph $G$ with $n$ vertices and $m$ edges.  We first decompose $G$
along its cutsets of size one (if any).  This can be done in time
$O(n+m)$ using depth-first search, see~\cite{tar:dfs}; depth-first
search produces the maximal $2$-connected subgraphs (``blocks'') of
$G$, and their number is at most $n$.  Clearly, $G$ contains a cycle
with a chord if and only if some block of $G$ contains a cycle with a
chord.  So our algorithm proceeds with examining the blocks of $G$.

Now suppose that $G$ is $2$-connected and has a proper $2$-cutset
$\{a, b\}$.  So $V(G)\setminus \{a, b\}$ can be partitioned into two
sets $X, Y$ such that there is no edge between them and each of $G[X
\cup \{ a, b \}]$ and $G[Y \cup \{ a, b \}]$ is not an $(a, b)$-path.
Let $G_X$ be the subgraph of $G$ induced by $X\cup \{a, b\}$ plus an
artificial vertex adjacent to $a$ and $b$, and define $G_Y$ similarly.
We consider that $G$ is decomposed into graphs $G_X$ and $G_Y$.  These
subgraphs can in turn be decomposed along proper $2$-cutsets.  

This is applied as long as possible, which yields a {\it proper
  $2$-cutset decomposition tree} $T_{2c}$ of $G$, whose leaves are
graphs that have no proper $2$-cutset.  By Theorem~\ref{thm:nochord},
if such a leaf contains no cycle with a chord then it is sparse, and
it is easy to see that the converse also holds.  So it suffices to
check that every leaf $L$ is sparse, which is easily done by examining
the degree of the two endvertices of every edge of $L$.

Exactly like in the previous section, a tree using 2-cutsets as we do
above has size $O(n)$.  Checking the leaves of the tree takes linear
time, so in total our algorithm runs in time $O(n^2m)$.

\begin{lemma}
  Recognizing a chordless graph can be performed in time $O(n^2m)$.
\end{lemma}

Note that chordless graphs are included in the class of graphs that do
not contain a cycle with a unique chord and that do not contain $K_4$.
These graphs are shown to be 3-colorable by a polynomial time
algorithm in \cite{nicolas.kristina:one}, but the proof is
complicated.  Here below, we show that this problem is very easy in
the particular case of chordless graphs.

\begin{lemma}
  A 2-connected chordless graph has a vertex of degree at most 2.  So,
  any chordless graph is 3-colorable and a 3-coloring can be found in
  linear time.
\end{lemma}

\begin{proof}
  If $G$ is chordless and 2-connected then it has an ear decomposition
  (see~\cite{bondy.murty:book}).  The last ear added to build $G$
  cannot be an edge because such an edge would be a chord of some
  cycle.  So, the last ear added to build $G$ is a path of length at
  least 2 and its interior vertices are of degree 2. 
\end{proof}

\subsection{Coloring \{ISK4, wheel\}-free graphs}
    
We present here a coloring algorithm which colors every \{ISK4,
wheel\}-free graph with three colors.  Its validity is based on
Theorem~\ref{l:WF3c} and it follows the same lines.  Let $G$ be any
\{ISK4, wheel\}-free graph with $n$ vertices and $m$ edges.

We first decompose $G$ along its clique-cutsets,  as in the preceding
subsection.  As in the proof of Theorem~\ref{l:WF3c},  a $3$-coloring
of the vertices of $G$ can be obtained simply by combining
$3$-colorings of each child of $G$ in the decomposition.  So let us
now suppose that $G$ has no clique cutset.

If $G$ contains a $K_{3, 3}$, then, by Lemma~\ref{l:decK33}, $G$ must
be a complete bipartite graph.  We can test that property in time
$O(n+m)$, and, if $G$ is complete bipartite, we return an obvious
$2$-coloring.  Now let us assume that $G$ contains no $K_{3, 3}$.

If $G$ has a proper $2$-cutset, then, as in the proof of
Theorem~\ref{l:WF3c}, we decompose $G$ into two graphs $G''_X$ and
$G''_Y$ and we can obtain a $3$-coloring of the vertices of $G$ by
combining $3$-colorings of $G''_X$ and $G''_Y$.  Moreover, we know
that these two graphs contain no ISK4, no wheel and no $K_{3, 3}$.
These graphs can be decomposed further (possibly also by clique
cutsets).  As above, one can prove that the total size of the
decomposition tree is $O(n)$ (we omit the details).

Finally, consider a leaf $L$ of the decomposition tree.  By
Theorem~\ref{th:nowheel}, $L$ is either a series-parallel graph, a
rich square, a line graph, or a complete bipartite graph.  Then
Lemmas~\ref{l:WF3c-sp}, \ref{l:WF3c-rs} and~\ref{l:WF3c-lg} show how
to construct a $3$-coloring of $L$ in polynomial time.  As for the
recognition, this can be implemented to run in time $O(n^2m)$.

\section*{Acknowledgement}

We are grateful to Alex Scott for pointing out the proof of
Theorem~\ref{th:ScottK4}.

\end{document}